\documentclass[11pt]{article}%
\usepackage{amsfonts}
\usepackage{amsmath}
\usepackage{amssymb}
\usepackage{amsxtra}
\usepackage{graphicx}
\usepackage{geometry}
\usepackage{color}
\usepackage{colortbl}
\usepackage{caption}%
\providecommand{\U}[1]{\protect\rule{.1in}{.1in}}
\newtheorem{theorem}{Theorem}

\newtheorem{corollary}[theorem]{Corollary}

\newtheorem{definition}[theorem]{Definition}

\newtheorem{lemma}[theorem]{Lemma}

\newtheorem{proposition}[theorem]{Proposition}
\newtheorem{remark}[theorem]{Remark}

\newcommand{\de}{\,\mathrm{d}}

\newenvironment{proof}[1][Proof]{\noindent\textbf{#1.} }{\ \rule{0.5em}{0.5em}}
\numberwithin{equation}{section}
\ifx\pdfoutput\relax\let\pdfoutput=\undefined\fi
\newcount\msipdfoutput
\ifx\pdfoutput\undefined\else
\ifcase\pdfoutput\else
\ifx\paperwidth\undefined\else
\ifdim\paperheight=0pt\relax\else\pdfpageheight\paperheight\fi
\ifdim\paperwidth=0pt\relax\else\pdfpagewidth\paperwidth\fi
\fi\fi\fi

\begin{document}

\title{Homogenization results for a coupled system of reaction-diffusion equations}
\author{G. Cardone\thanks{Universit\`{a} del Sannio, Department of Engineering, Corso
Garibaldi, 107, 82100 Benevento, Italy; email: giuseppe.cardone@unisannio.it},
C. Perugia\thanks{University of Sannio, DST, Via Port'Arsa 11, 82100 Benevento,
Italy; email: cperugia@unisannio.it}, C. Timofte\thanks{University of Bucharest, Faculty of Physics, Bucharest-Magurele,  P.O. Box MG-11, Romania; email:
claudia.timofte@g.unibuc.ro.}}
\maketitle

\begin{abstract}
\medskip The macroscopic behavior of the solution of a coupled system of partial differential equations arising in the modeling of reaction-diffusion processes in periodic porous media is analyzed. Our mathematical model can be used for studying several metabolic processes taking place in li\-ving cells, in which biochemical species can diffuse in the cytosol and react both in the cytosol and also on the organellar membranes. The coupling of the concentrations of the biochemical species is realized via various properly scaled nonlinear reaction terms. These nonlinearities, which model, at the microscopic scale, various volume or surface  reaction processes, give rise in the macroscopic model to different effects, such as the appearance of additional source or sink terms or of a non-standard diffusion matrix.\\

{\bf Keywords}: homogenization; nonlinear flux conditions; reaction-diffusion equations.

\medskip

{\bf AMS subject classifications}: 35B27; 35K57; 35Q92.

\end{abstract}

\section{Introduction\label{sect1}}

The goal of this paper is to analyze, via periodic homogenization techniques, the macroscopic behavior of the solution of a coupled system of partial differential equations arising in the modeling of reaction-diffusion phenomena in periodic porous media. Our results can be useful for studying some metabolic processes occurring in biological cells, in which biochemical species can diffuse in the cytosol and react both in the cytosol and also  on the membranes of various organelles,  which are present in the cytoplasm, such as chloroplasts, endoplasmic reticulum or mitochondria.

From a mathematical point of view, we consider a perforated domain $\Omega^{\ast}_\varepsilon$, obtained by removing from a smooth domain $\Omega$ a set of periodically distributed inclusions, $\varepsilon$ being a positive small parameter related to the size and the periodicity of the structure. For applications in biology, the domain $\Omega$ can be seen, in a simplified setting, as the cytoplasm of a biological cell, consisting of the cytosol $\Omega^{\ast}_\varepsilon$ and of various organelles $S^\varepsilon$, periodically distributed, with period $\varepsilon$, in the cytoplasm (see Section \ref{Sec2}). In such a geometry, we consider, at the microscale, a coupled system of three reaction-diffusion equations, governing the evolution of the concentration of three types of biochemical species (metabolites) in the perforated domain (cytosol), with suitable nonlinear boundary conditions at the surface of the inclusions (at the organellar membranes) and initial conditions. On the outer cellular membrane $\partial \Omega$, we impose Dirichlet conditions (see system \eqref{eqmicro} and \cite{DV}).

The model assumes that the coupling between the concentrations of the biochemical species is realized through nonlinear reaction terms, appearing in the perforated domain and at the boundary of the inclusions, as well. In particular, the coupling phenomenon at the boundary of the inclusions is described by imposing different nonlinear flux conditions for the three concentration fields. The nonlinear functions appearing in these flux conditions depend on the concentrations of the three species and are properly scaled, depending on the processes occurring at the membranes of the cellular organelles.
The nonlinearities of our model were inspired by kinetics corresponding to multi-species enzyme catalyzed reactions, which are generalizations of the classical Michaelis-Menten kinetics to multi-species reactions.  A rigorous mathematical model for such processes is analyzed in \cite{Gahn}$\div$\cite{Gahn3}, but with a different scaling and taking into account more complex phenomena.\\
The main mathematical difficulties behind our study come from the strong coupling between the equations governing the evolution of the concentrations of the biochemical species and, also, from the particular scaling of the nonlinear boundary terms. As a matter of fact, the novelty of our paper consists exactly in the special coupling and scaling of the boundary terms. The form of the nonlinearities arising at the microscale and this non-standard coupling lead, at the macroscale, to the appearance of additional source or sink terms and of a non-standard diffusion matrix, which is not constant (see Theorem \ref{teounf} and Corollary \ref{teohom}). Hence, we show that some fast reactions occurring, at the microscale, on the boundaries  of  $S^\varepsilon$ (see the second equation in system \eqref{eqmicro}) could lead, at the macroscale, to a cross - diffusion system (see for this terminology the recent papers \cite{JP} and \cite{J}).\\

The metabolic pathway which our system better fits is the fatty acid activation on the outer mitochondrial membrane (see \cite{Enz}). Such a reactive system involves three metabolites, a fatty acid, the adenosine triphosphate (ATP) and the coenzyme A (CoA), which can diffuse in the cytosol and can also react on the mitochondrial membrane with different rate, with ATP being the slowest one. They can react in the cytosol thanks to the enzyme acyl coa - synthetase present on the outer mitochondrial membrane. Since this metabolic pathway is strongly ATP-dependent, the concentration of ATP itself limits the mithocondrial availability of fatty acid and CoA. This aspect could well explain why, in our limit problem, the effective diffusion matrices of the fastest metabolites depend on the concentration of the slowest one. Indeed, a shortage of ATP in the cytosol would imply the stop of the enzyme catalyzed reaction previously described limiting the diffusion of the fatty acid and the CoA in the whole cell.\\

Our results might find some applicability also in the case of a porous medium in which various chemical substances are allowed to diffuse and to react inside and on the walls of the porous medium and in multiscale modeling of colloidal dynamics in porous media (see, for instance, \cite{Hornung} and \cite{KMK}).\\

As already mentioned, the problem we consider here comprises nonlinear terms, properly (differently) scaled, asking, therefore, for different  analytical techniques. For obtaining our macroscopic model, we use suitable extension operators (see \cite{Cio-Pau}, \cite{Hopker}, \cite{Bohm}, and \cite{Meirmanov}) and the periodic unfolding method, adapted to time-dependent functions (see, for instance, \cite{Cio-Dam-Don-Gri-Zaki}, \cite{Cio-Dam-Gri1}, \cite{Cioranescu-Donato-Zaki}, \cite{Amar1}, and \cite{Amar2}).

Related homogenization problems for reaction-diffusion problems in porous media were addressed, for instance, in \cite{Hornung}, \cite{Hor-Jag}, \cite{Jager}, \cite{Neu-Jag}, \cite{Muntean1}, \cite{Malte}, \cite{Fat-Mun-Pta}, \cite{Gahn}, \cite{Gahn2}, \cite{Gahn3}, \cite{Graf}, \cite{Mar-Pta}, \cite{Pop1}, \cite{Pop2}, \cite{Timofte1}, and \cite{Timofte2}. Problems involving fast surface reaction terms can be encountered, for instance, in \cite{Amar},  \cite{Allaire}, \cite{Allaire1}, and \cite{Iji-Mun}. As far as we know, there are only a few recent papers in which homogenized matrices depending on the solution itself appear. Non-constant homogenized matrices in which a nonlinear function depending on the solution of the macroscopic problem appears can be found, in different contexts from ours and for particular types of nonlinearities, in the recent papers \cite{Allaire}, \cite{Iji-Mun}, and \cite{Bun-Tim}. For the case in which the homogenized matrix depends on the solution of the limit problem, we also refer the reader to \cite{Allaire1}, \cite{Allaire2}, \cite{Allaire3}, \cite{Timofte2}, and \cite{DLN}.

The rest of the paper is structured as follows: in Section \ref{Sec2}, we describe the geometry of the periodic porous medium and we set the microscopic problem. Existence and uniqueness results and {\it a priori} estimates for the solution of our microscopic problem are obtained in Section \ref{Sec3}. The main homogenization results of this paper are stated and proven in Section \ref{Sec4}.

\section{Setting of the problem\label{Sec2}}

We start this section by describing the geometry of the problem. We assume that the porous medium possesses a periodic microstructure. Let $Y =(0, 1)^n$ be the reference cell and $S$ be a closed strict subset of $\overline{Y}$ with the boundary $\Gamma:=\partial S$ being Lipschitz continuous. We restrict ourselves to the case in which $S$ is connected, but this assumption can be easily removed, since we can deal with the case in which $S$ has a finite number of connected components. Let $Y^*:=Y\setminus S$. The sets $S$ and $Y^*$ will be the reference inclusion and the perforated cell, respectively. Now, let us consider a bounded connected smooth open set $\Omega$ in $\Bbb R^n$, with a Lipschitz boundary $\partial \Omega$ (we shall be interested in the physically relevant cases $n=2$ or $n=3$). Let us notice that, for simplicity, we restrict ourselves here to a case in which the domain $\Omega$ is supposed to be representable by a finite union of axis-parallel cuboids with corner coordinates in $\Bbb Z^n$ (see \cite{Gra-Pet}, \cite{Bohm}, and \cite{Hopker}), but the results given in this paper still hold true for more general domains $\Omega$ (see \cite{Bohm} and \cite{Hopker}). Let $\varepsilon\in (0,1)$ be a small positive parameter, related to the characteristic dimensions of the structure and which takes values in a sequence of strictly positive numbers tending to zero, such that  the stretched domain $\varepsilon ^{-1} \Omega$ can be represented as a finite union of axis-parallel cuboids having corner coordinates in ${\mathbb Z}^n$. For each ${\bf k}\in \Bbb Z^n$, let $Y_{\bf k}={\bf k}+Y^*$, $\Gamma_{\bf k}:={\bf k}+\Gamma$, and $K_\varepsilon:=\{{\bf k}\in \Bbb Z^n \; | \; \varepsilon Y_{\bf k} \subset \Omega\}$. Then, setting $S^\varepsilon:=\bigcup\limits_{{\bf k}\in K_\varepsilon}\varepsilon({\bf k}+S)$, the perforated domain $\Omega^{\ast}_{\varepsilon}$ is obtained by removing from $\Omega$ the set of inclusions $S^\varepsilon$, i.e. $\Omega^{\ast}_\varepsilon:=\Omega\setminus S^\varepsilon$. Moreover, if $\Gamma^\varepsilon:=\bigcup\limits_{{\bf k}\in K_\varepsilon}\varepsilon\Gamma_{\bf k}$ denotes the inner boundary of the porous medium, it holds that $\Gamma^\varepsilon\cap \partial \Omega=\emptyset$.

In such a periodic microstructure, we shall consider a coupled system of nonlinear reaction-diffusion equations, with suitable boundary and initial conditions. The coupling phenomenon at the boundary of the inclusions is described by imposing suitably scaled nonlinear flux conditions for the concentrations. More precisely, if we denote by $[0,T]$, with $T>0$, the time interval of interest, we shall analyze the effective behavior, as the small parameter $\varepsilon \rightarrow 0$, of the solution  $(c^{\varepsilon}_1, c^{\varepsilon}_2, c^\varepsilon_3)$ of the following system:
\begin{equation}
\left\{
\begin{array} {ll}
\partial_{t}c_{i}^{\varepsilon}-\hbox{div} (D_i^\varepsilon \nabla c_{i}^{\varepsilon
})=F_{i}^{\varepsilon}%
(x,c_{1}^{\varepsilon},c_{2}^{\varepsilon}, c_{3}^{\varepsilon})&\text{ in
}(0,T)\times\Omega^{\ast}_{\varepsilon},\,\, i\in \{1,2,3\}\\
\\
D_i^\varepsilon \nabla c_{i}^{\varepsilon}\cdot\nu^\varepsilon=\displaystyle \frac{1}{\varepsilon}\, G_{i}(c_{1}^{\varepsilon}%
,c_{2}^{\varepsilon}, c_{3}^{\varepsilon})&\text{ on }(0,T)\times
\Gamma^{\varepsilon},\, \, i\in \{1,2\}\\
\\
D_3^\varepsilon \nabla c_{3}^{\varepsilon}\cdot\nu^\varepsilon=\varepsilon\, G_{3}^{\varepsilon}(x,c_{1}^{\varepsilon}%
,c_{2}^{\varepsilon}, c_{3}^{\varepsilon})&\text{ on }(0,T)\times
\Gamma^{\varepsilon}\\
\\
c_{i}^{\varepsilon}=0 & \text{ on }(0,T)\times\partial\Omega,\, \,  i\in \{1,2,3\}\\
\\
c_{i}^{\varepsilon}(0)=c_{i}^0 & \text{ in }\Omega^{\ast}_{\varepsilon},\, \,  i\in \{1,2,3\},
\end{array}
\right.  \label{eqmicro}%
\end{equation}
where $\nu^\varepsilon$ is the unit outward normal to $\Gamma^{\varepsilon}$.\\
In \eqref{eqmicro}, the concentrations $c_{i}^{\varepsilon}$ of the chemical species (metabolites), the diffusion matrix $D^\varepsilon$ and the nonlinear reactions rates $F_{i}^{\varepsilon}$, $i\in \{1,2,3\}$, $G_{i}$, $i\in \{1,2\}$, and $G_{3}^{\varepsilon}$ satisfy suitable conditions. More precisely, we make the following assumptions on the data (see also \cite{Gahn, Conca-Diaz-Timofte, DLN}): \\
\\
$(\mathbf{H}_1)$ For $i\in \{1,2,3\}$, the diffusion matrices are given by $D_i^\varepsilon(x)=D_i\left(\dfrac{x}{\varepsilon}\right)$, where
\begin{itemize}
\item [1.] $D_i\in (L^\infty(Y))^{n\times n}$ is a  $Y$-periodic and symmetric matrix-field;
\item [2.] for any $\xi\in\mathbb{R}^{n}$, $\alpha |\xi|^2 \leq (D_i(y)\xi, \xi) \leq \beta |\xi|^2$, almost everywhere in $Y$, for some positive real constants $\alpha$ and $\beta$.
\end{itemize}
$(\mathbf{H}_2)$ For $i\in \{1,2,3\}$, the reaction rates are given by $F_{i}^{\varepsilon}(x,s_{1},s_{2},s_{3})=F_i(\frac
{x}{\varepsilon},s_{1},s_{2},s_{3})$, where $F_i:\mathbb{R}^{n}\times\mathbb{R}^{3}\rightarrow\mathbb{R}$ satisfies the following hypotheses:
\begin{itemize}
\item [1.] $F_i$ is continuous;
\item [2.] $F_i(\cdot,s)$ is a $Y$-periodic function for all $s\in \mathbb{R}^{3}$;
\item [3.] $F_i(y,\cdot)$ is Lipschitz continuous for all $y\in \mathbb{R}^{n}$ with constant independent of $y$;
\item [4.] $F_i(y, 0)=0$ for all $y\in \mathbb{R}^n$.
\end{itemize}
\noindent $(\mathbf{H}_3)$ For $i\in \{1,2\}$, the nonlinear functions in the first flux condition are given by
\begin{equation}\label{formG}
G_{i}(s_{1},s_{2},s_{3})=(-1)^i \left (s_1-s_2\right ) \, H \left (s_{3}\right),
\end{equation}
where $H: \Bbb R \rightarrow \Bbb R^*_+$ satisfies the following hypotheses:
\begin{itemize}
\item [1.] there exists a positive constant $l$ such that
\begin{equation*}
0\leq H(s)\leq l<\infty \text{  for all } s\in\Bbb R;
\end{equation*}
\item [2.] $H(0)=0$;
\item [3.] $H\in C^1(\mathbb{R})$;
\item [4.] there exists a positive constant $L$ such that
\begin{equation*}
\vert H^{'}(s)\vert \leq L<\infty \text{  for all } s\in\Bbb R.
\end{equation*}
\end{itemize}
$(\mathbf{H}_4)$ The nonlinear function in the second flux condition is given by $G_{3}^{\varepsilon}(x,s_{1},s_{2},s_{3})=G_3\left (\frac
{x}{\varepsilon},s_{1},s_{2},s_{3}\right)$, where $G_3: \mathbb{R}^{n}\times\mathbb{R}^{3}\rightarrow\mathbb{R}$ satisfies the following hypotheses:
\begin{itemize}
\item [1.] $G_3$ is continuous;
\item [2.] $G_3(\cdot,s)$ is a $Y$-periodic function for all $s\in \mathbb{R}^{3}$;
\item [3.] $G_3(y,\cdot)$ is Lipschitz continuous for all $y\in \mathbb{R}^{n}$ with constant independent of $y$;
\item [4.] $G_3(y, 0)=0$ for all $y\in \mathbb{R}^n$.
\end{itemize}
\bigskip
\noindent $(\mathbf{H}_5)$ Set $\left( \cdot\right)_- =\min\left\{
\cdot,0\right\}$. We assume that for a.e. $(x,s)\in \mathbb{R}^n\times\mathbb{R}^3$ it holds
\begin{equation}\label{assF2}
\begin{array}{l}
F_{1}^{\varepsilon}(x,s_{1},s_{2},s_{3})\left(  s_{1}\right)  _{-}
+F_{2}^{\varepsilon}(x,s_{1},s_{2},s_{3})\left(  s_{2}\right)  _{-} +F_{3}^{\varepsilon}(x,s_{1},s_{2},s_{3})\left(  s_{3}\right)  _{-}\\
\\
\leq
C_1\left(  \left\vert \left(  s_{1}\right)  _{-}\right\vert ^{2}+\left\vert
\left(  s_{2}\right)  _{-}\right\vert ^{2}+\left\vert
\left(  s_{3}\right)  _{-}\right\vert ^{2}\right)
\end{array}
\end{equation}
and
\begin{equation}\label{assG}
G_{3}^{\varepsilon}(x,s_{1},s_{2},s_{3})\left(  s_{3}\right)  _{-} \leq
C_2\left(  \left\vert \left(  s_{1}\right)  _{-}\right\vert ^{2}+\left\vert
\left(  s_{2}\right)  _{-}\right\vert ^{2}+\left\vert
\left(  s_{3}\right)  _{-}\right\vert ^{2}\right),
\end{equation}
with $C_1$ and $C_2$ positive constants independent of $\varepsilon$.\\
\\
$(\mathbf{H}_6)$ There exist $\Lambda>0$ and $A>0$ such that (see \cite{Gahn3}), for $i\in \{1,2,3\}$,
\begin{equation}\label{assFG_0}
\left\{
\begin{array}{l}
F_i(\cdot, s) \leq A s_i \quad \textrm{ for all   } s\in \Bbb R^3 \textrm{   with   } s_i \geq \Lambda,\\
\\
G_3(\cdot, s) \leq A s_3 \quad \textrm{ for all   } s\in \Bbb R^3 \textrm{   with   } s_3 \geq \Lambda.\\
\end{array}
\right.
\end{equation}
$(\mathbf{H}_7)$ The initial concentrations $c_{i}^0\in L^{2}(\Omega)$,
$i\in \{1, 2, 3\}$, are assumed to be non-negative and bounded independently with respect
to $\varepsilon$ by a constant $\Lambda$ (the same constant as the one in assumption $(\mathbf{H}_6)$, i.e. for $i\in \{1,2,3\}$ it holds
\begin{equation}\label{binitial}
0\leq c_i^0\leq \Lambda.
\end{equation}

From the Lipschitz continuity, we obtain that for all $(x,s)\in \mathbb{R}^n\times \mathbb{R}^3$ the following growth conditions hold:
\begin{equation}\label{F2}
\left\vert F_i^\varepsilon (x,s_{1},s_{2},s_{3})\right\vert \leq C_3(1+|s_{1}%
|+|s_{2}|+|s_{3}| ),\,\,i\in \{1,2,3\}
\end{equation}
and
\begin{equation}\label{assG1}
\left\vert G_3^\varepsilon (x,s_{1},s_{2},s_{3})\right\vert \leq C_4(1+|s_{1}%
|+|s_{2}|+\left\vert s_{3}\right\vert ),
\end{equation}
with $C_3$ and $C_4$ positive constants independent of $\varepsilon$.\\
\smallskip
Moreover by $(\mathbf{H}_2)_4$ and $(\mathbf{H}_4)_4$, \eqref{F2} and \eqref{assG1} become
$$\vert F_i^\varepsilon (x,s_{1},s_{2},s_{3})\vert \leq C_5(|s_{1}
|+|s_{2}|+|s_{3}| ),\,\,i\in \{1,2,3\}
$$
and
$$
\vert G_3^\varepsilon (x,s_{1},s_{2},s_{3})\vert \leq C_6(|s_{1}
|+|s_{2}|+|s_{3}| ).
$$

\begin{remark}\label{time} For simplicity, we assume that all the parameters involved in our model are time independent, but the case in which they depend on time can be also treated.
\end{remark}

\begin{remark}\label{remF}
We point out that we don't have too much freedom for our nonlinearities and for their scalings, since, when working with systems, it is difficult to ensure positivity, essential boundedness and, most important, strong convergence results.
\end{remark}

\begin{remark}\label{remG}
The adsorption/desorption processes and the reactions taking place at the surface of the inclusions depend on many factors, such as the nature of the molecules (various molecules can have very low or very high adsorption, some molecules can adsorb and react with other molecules, which do not adsorb), the value of the concentrations of the involved species, the nature of the solid surface, the total surface area of the adsorbent, etc. So, various surface kinetics corresponding to multi-species enzyme catalyzed reactions, such as Langmuir or Eley-Rideal mechanisms, can be considered in order to describe the nonlinear fluxes at the surface of the inclusions. The fact that the reactions of the  biochemical species at the surface $\Gamma^\varepsilon$ are assumed to be fast or slow is reflected in the different scaling of the corresponding fluxes in equations \eqref{eqmicro}$_{2,3}$. Such a specific scaling will lead, at the macroscale, to different effects, namely a non-standard homogenized matrix and an additional source or sink term, respectively (see Theorem \ref{teounf} and Corollary \ref{teohom}).
\end{remark}
\textit{Example 1.}\,
Let us give here some concrete examples of nonlinear functions $F_i$, for $i\in \{1,2,3\}$ (see, also, \cite{Gahn}, Section 5 in \cite{Gahn2}, and Chapter 4 in \cite{Gahn3}).  \\
If $a_0(y)$ is a positive bounded measurable $Y$ -periodic function, $c_k>0$ and $k\in \{0,1\}^3$, then 
\[
F_i(y,s_1, s_2, s_3)= a_0(y)\displaystyle \frac{s_1 s_2 s_3}{\sum_{\vert k\vert \leq 3}c_k s^k}.
\]
For simplicity, we can take all the coefficients in front of our variables to be equal to one. Thus, we consider functions $F_i$, for $1\leq i\leq 3$, of the following form:
\[
F_i(y,s_1, s_2, s_3)= a_0(y) \displaystyle \frac{s_1 s_2 s_3}{1+s_1+s_2+s_3+ s_1 s_2 +s_1s_3+s_2s_3+s_1 s_2 s_3}.
\]
Other possible concrete examples could be:
\[
F_1(y, s_1, s_2, s_3)=a_0 (y) \, \displaystyle \frac{s_1 s_2}{1+s_1+s_2+s_1 s_2 +s_1 s_2 s_3}+s_1,
\]
\[
F_2(y, s_1, s_2, s_3)=\displaystyle \frac{s_2 s_3}{1+s_2+s_3+s_2s_3+s_1 s_2 s_3},
\]
\[
F_3(y, s_1, s_2, s_3)=\displaystyle \frac{s_1 s_3}{1+s_1+s_3+s_1s_3+s_1 s_2 s_3}+s_3.
\]

A very simplified setting could be to assume that each $F_i$ depends only on the corresponding concentration $s_i$ and it is given by the Michaelis-Menten kinetics arising in biochemistry when dealing with enzyme-catalyzed reactions. \\

\textit{Example 2.}\,
For $G_3$ we might take, for instance, kinetics similar to the ones used for $F_i$ or kinetics of the form
\[
G_3(s_1, s_2, s_3)= a(y) \dfrac{s_1^4}{1+s_1^4} \, s_3 +s_3,
\]
where $a(y)$ is a positive bounded measurable $Y$ -periodic function.\\

\textit{Example 3.}\,
As an example for a nonlinear function $H$ in \eqref{formG}, we can consider the Langmuir kinetics, i.e.
\begin{equation} \label{exampleH}
H(s)=\displaystyle \frac{a s}{1+b s}, \quad a, b>0.
\end{equation}
Thus, for a concrete relevant example of nonlinear functions $G_i$, for $i\in \{1,2\}$, we have
\begin{equation}\label{exGi}
G_i(s_1, s_2, s_3)=(-1)^i a \dfrac{(s_1-s_2)s_3}{1+b s_3}.
\end{equation}
We remark, that with this choice, the boundary term involving function $H$ is well-defined. Indeed, the function  in \eqref{exGi} satisfies assumption $(\mathbf{H}_3)$ since the concentration fields $c_i^\varepsilon$, for $i\in \{1,2,3\}$, are positive and essentially bounded, as we shall prove in Section \ref{Sec3}.\\\\

For negative values of variables, in the above functions singularities may appear. To handle this type of nonlinearity, as in \cite{Gahn}$\div$\cite{Gahn3}, we first consider the modified kinetics with the modulus appearing in all the terms in the denominator. After proving existence and uniqueness results for the problem with
the modified kinetics, we prove that the solution is nonnegative, and, so, it is also a
solution of the original problem (see, for instance, Example 1 in [24]).\\

\begin{remark}\label{rem2}

Actually, having in mind \eqref{exampleH}, it would be enough to put $0<H'\leq L$ in order to consider a one to one and increasing function  as done for instance in \cite{Allaire}.  In this sense, our assumption $(\mathbf{H}_3)_4$ is more general from the mathematical point of view.  
\end{remark}

The assumptions we made are natural structural assumptions. Most of them are mathematical requirements, which, however, have a strong physical justification, as underlined in the following remarks.

\begin{remark}\label{remhyp1} Hypotheses $(\mathbf{H}_1)$, $(\mathbf{H}_2)_{1,2,3}$, $(\mathbf{H}_3)$, and $(\mathbf{H}_4)_{1,2,3}$ are, on one hand, mathematical and are needed for proving existence and uniqueness results and estimates for the microscopic model. In general, apart from linear functions, one has to assume Lipschitz or monotonicity conditions to ensure well-posedness. $(\mathbf{H}_2)_{1,2,3}$ is also needed to prove strong convergence. However, $(\mathbf{H}_2)$ and $(\mathbf{H}_4)$ are standard hypotheses used for this kind of problems in the literature, describing chemical reactions which are relevant in many concrete applications. In biology and in chemistry, many important processes behave like having a monotone or a Lipschitz character and, so, our hypotheses on the nonlinear functions $F_i$ and $G_i$ are physically justified. Also, the regularity conditions imposed for our nonlinearities are technical requirements. This kind of regularity can probably be improved. The conditions asked in $(\mathbf{H}_3)$ for $G_i$ are not standard, but such nonlinearities are still relevant for describing several chemical reactions, such as those implied in the Eley-Rideal mechanisms. 
We notice that $(\mathbf{H}_2)_4$ and $(\mathbf{H}_4)_4$ are just physical assumptions: in absence of
concentrations, we do expect no reaction to occur. These conditions are not relevant from a mathematical
point of view. Technically, such conditions can be removed.\\
\end{remark}

\begin{remark}\label{remhyp2} Hypotheses $(\mathbf{H}_5)$, $(\mathbf{H}_6)$, $(\mathbf{H}_7)$ and again $(\mathbf{H}_3)$ are needed for ensuring the
non-negativity and the essential boundedness of the solutions, such properties being, in fact, reasonable from a biological point of view (see, for example, \cite{Gahn2} and \cite{Gahn3}). For instance, $(\mathbf{H}_5)$ is a technical requirement and it ensures the positivity of our concentrations. However, we mention that other choices are still possible. In fact, we could ask, as in \cite{HK}, that $F_i(\cdot, s)\geq 0$, for all $s=(s_1, s_2, s_3)$ with $s_i\leq 0$. This means, in fact, that if a chemical species vanishes, then it
can only be produced.
\end{remark}

The main difficulties behind our study come from the strong coupling between the equations governing the evolution of the concentration fields, from the nonlinearity of the model and from the interesting interface phenomena (which are far from being perfectly understood) appearing in the heterogeneous media under consideration.

\section{The microscopic model}\label{Sec3}

\subsection{Existence and uniqueness for the microscopic model}

The following lemma will be used several times throughout the paper (see \cite[Lemma 5]{Gahn}).
\begin{lemma}\label{lemtrace}
(i) Let $\Omega$ be an arbitrary Lipschitz domain and $\delta>0$. Then,
\begin{equation}\label{trace1}
\Vert u\Vert^2_{L^2(\partial \Omega)}\leq C_{\delta} \Vert u\Vert^2_{L^2(\Omega)}+\delta \Vert \nabla u\Vert^2_{L^2(\Omega)},
\end{equation}
for all $u\in H^1(\Omega)$ and with a constant $C_\delta>0$ depending on $\delta$.\\
(ii) Let $p\in [1,+\infty)$ and $u_\varepsilon \in L^p(\Omega^{\ast}_{\varepsilon})$. Since the trace operator from $W^{1,p}(Y^*)$ into $L^p(\Gamma)$ is continuous, one has
\begin{equation}\label{trace2}
\varepsilon\Vert u^\varepsilon\Vert^p_{L^p(\Gamma^\varepsilon)}\leq C\left( \Vert u^\varepsilon\Vert^p_{L^p(\Omega^{\ast}_{\varepsilon})}+\varepsilon^p\Vert \nabla u^\varepsilon \Vert^p_{L^p(\Omega^{\ast}_{\varepsilon})}\right).
\end{equation}
(iii) Under the same assumptions as in (ii) and for an arbitrary $\delta>0$, we have
\begin{equation}\label{trace3}
\varepsilon\Vert u^\varepsilon\Vert^p_{L^p(\Gamma^\varepsilon)}\leq C_\delta \Vert u^\varepsilon\Vert^p_{L^p(\Omega^{\ast}_{\varepsilon})}+\delta\varepsilon^p \Vert \nabla u^\varepsilon \Vert^p_{L^p(\Omega^{\ast}_{\varepsilon})}.
\end{equation}
\end{lemma}
\bigskip
For giving the definition of a weak solution of the system \eqref{eqmicro}, following \cite{Gahn}, for any Banach space $V$, we denote its dual by $V'$ and we introduce the space
\[
\mathcal{W}\left(0,T;V,V'\right):=\{u\in L^2\left(0,T;V\right):\partial_t u\in L^2\left(0,T;V'\right)\},
\]
where the time derivative $\partial_t u$ is understood in the distributional sense. It is a Banach space if it is endowed with the norm of the graph
\[
\|u\|_{\mathcal{W}}:=\|u\|_{L^2(0,T;V)} + \|\partial_t u\|_{L^2(0,T;V')}.
\]
We denote
\[
H_{\partial \Omega}^{1}(\Omega^{\ast}_{\varepsilon})= \{ \, v\in H^{1}(\Omega^{\ast}_{\varepsilon} )
\, \mid \, v=0 \textrm{ on } \partial \Omega \}.
\]
\begin{definition}
We say that $(c_{1}^{\varepsilon},\,c_{2}^{\varepsilon}, \, c_{3}^{\varepsilon})\in \mathcal{W}\left(0,T;H_{\partial \Omega}^{1}(\Omega^{\ast}_{\varepsilon}),(H_{\partial \Omega}^{1}(\Omega^{\ast}_{\varepsilon}))'\right)^3$
is a weak solution of  problem \eqref{eqmicro} if for any $v\in H_{\partial \Omega}^{1}(\Omega^{\ast}_{\varepsilon})$ and for a.e. $t\in(0,T)$ it holds that, for $i\in \{1,2\}$,
\begin{equation}\label{weakconcentration}
\begin{array}{l}
\displaystyle\langle\partial_{t}c_{i}^{\varepsilon},v\rangle_{\Omega^{\ast}_{\varepsilon}}+\int_{\Omega^{\ast}_{\varepsilon}}D_i^\varepsilon \nabla c_{i}^{\varepsilon}\cdot\nabla v\mathrm{d}x  =\int_{\Omega^{\ast}
_{\varepsilon}}F_{i}^{\varepsilon}(x,c_{1}^{\varepsilon},c_{2}^{\varepsilon
},c_{3}^{\varepsilon})v\,\mathrm{d}x+\displaystyle \frac{1}{\varepsilon}\int_{\Gamma^{\varepsilon}}G_{i}(c_{1}%
^{\varepsilon},c_{2}^{\varepsilon},c_{3}^{\varepsilon})v\,\mathrm{d}\sigma
_{x}\\
\\
\displaystyle\langle\partial_{t}c_{3}^{\varepsilon},v\rangle_{\Omega^{\ast}_{\varepsilon}}+\int_{\Omega^{\ast}_{\varepsilon}}D_3^\varepsilon \nabla c_{3}^{\varepsilon}\cdot\nabla v\mathrm{d}x  =\int_{\Omega^{\ast}
_{\varepsilon}}F_{3}^{\varepsilon}(x,c_{1}^{\varepsilon},c_{2}^{\varepsilon
},c_{3}^{\varepsilon})v\,\mathrm{d}x + \displaystyle \varepsilon\int_{\Gamma^{\varepsilon}}G_{3}^{\varepsilon}(x,c_{1}%
^{\varepsilon},c_{2}^{\varepsilon},c_{3}^{\varepsilon})v\,\mathrm{d}\sigma
_{x},
\end{array}
\end{equation}
together with the initial conditions
\begin{equation}\label{weakinitial}
c_{i}^{\varepsilon}(0)=c_{i}^0\quad\text{ in }\Omega^{\ast}_{\varepsilon}, \quad i\in \{1,2,3\}.
\end{equation}
\end{definition}
Here, we denoted by $\langle\cdot,\cdot\rangle_{\Omega^{\ast}_{\varepsilon}}$ the duality pairing $\langle\cdot,\cdot\rangle_{H_{\partial \Omega}^1(\Omega^{\ast}_{\varepsilon}),(H_{\partial \Omega}^1(\Omega^{\ast}_{\varepsilon}))'}$ of $H_{\partial \Omega}^1(\Omega^{\ast}_{\varepsilon})$ with its dual space $(H_{\partial \Omega}^1(\Omega^{\ast}_{\varepsilon}))'$.
\begin{theorem}\label{teoexun} There exists a unique weak solution $(c_{1}^{\varepsilon},\,c_{2}^{\varepsilon}, \,c_{3}^{\varepsilon})$ of problem \eqref{eqmicro}.
\end{theorem}
\begin{proof}
We can argue in a similar manner as in \cite{Gahn} (see, also \cite{Gahn3}, \cite{Gra-Pet}, and \cite{Graf}). Indeed, we can use Schaefer's fixed point theorem. To this end, following \cite{Gahn}, one can define the fixed point operator
${\cal F}: X \rightarrow X$, where $X =L^2\left(0,T;H^{\beta}_{\partial \Omega}(\Omega^{\ast}_{\varepsilon})\right)^3$, for $\beta\in (\frac{1}{2},1)$. For $\widehat{c}^\varepsilon=(\widehat{c}_{1}^\varepsilon,\,\widehat{c}_{2}^\varepsilon, \, \widehat{c}_{3}^\varepsilon)\in X$, let ${\cal F}(\widehat{c}^\varepsilon) =c^{\varepsilon}= (c_1^{\varepsilon}, \, c_2^{\varepsilon}, \, c_3^{\varepsilon})$, where $c^{\varepsilon}\in \mathcal{W}\left(0,T;H^{1}_{\partial \Omega}(\Omega^{\ast}_{\varepsilon}),(H^{1}_{\partial \Omega}(\Omega^{\ast}_{\varepsilon}))'\right)^3$ is the unique solution of the linearization of the system \eqref{eqmicro} (i.e. in \eqref{eqmicro} we replace, for $i\in \{1,2,3\}$, $F_i^\varepsilon (x, c_{1}^{\varepsilon},\,c_{2}^{\varepsilon}, \,c_{3}^{\varepsilon})$ by $F_i^\varepsilon (x, \widehat{c}_{1}^\varepsilon,\,\widehat{c}_{2}^\varepsilon, \, \widehat{c}_{3}^\varepsilon)$, $G_i(c_{1}^{\varepsilon},\,c_{2}^{\varepsilon}, \,c_{3}^{\varepsilon})$ by $G_i(\widehat{c}_{1}^\varepsilon,\,\widehat{c}_{2}^\varepsilon, \, \widehat{c}_{3}^\varepsilon)$, for $i\in \{1,2\}$, and $G_3^\varepsilon(x, c_{1}^{\varepsilon},\,c_{2}^{\varepsilon}, \,c_{3}^{\varepsilon})$  by  $G_3^\varepsilon (x, \widehat{c}_{1}^\varepsilon,\,\widehat{c}_{2}^\varepsilon, \, \widehat{c}_{3}^\varepsilon)$, respectively). Due to the fact that the embedding
\[
\mathcal{W}\left(0,T;H^{1}_{\partial \Omega}(\Omega^{\ast}_{\varepsilon}),(H^{1}_{\partial \Omega}(\Omega^{\ast}_{\varepsilon}))'\right)^3 \hookrightarrow X
 \]
is compact, using estimates similar to those in Lemma \ref{lemest} below, one gets that the operator ${\cal F}$ is continuous and compact and the set $\{c\in X \, \vert \, c=\lambda {\cal F}(c), \, \textrm{for some } \lambda \in [0,1]\}$ is bounded in $X$. Then, Schaefer's theorem gives the existence of a weak solution of problem \eqref{eqmicro}. As in \cite{Gahn}, the uniqueness follows from the Lipschitz continuity of the functions $F_i^\varepsilon$ and $G_i^\varepsilon$ and Gronwall's inequality.
\end{proof}
\bigskip
\\
In what follows, we prove the nonnegativity and the uniform boundedness from above of the concentration fields $c_i^\varepsilon$, for $i\in \{1,2,3\}$, which is a reasonable condition from the point of view of applications in biology. In order to show nonnegativity, we have to take into account the fact that, for $i\in \{1,2,3\}$, the generalized time derivative $\partial_t c_i^\varepsilon$ is only an element of $L^2(0,T;((H^1_{\partial \Omega}(\Omega^{\ast}_{\varepsilon}))')$ and we don't know if the time derivative of $(c_i^\varepsilon)_{-}$ exists or not. So, as in \cite{Gahn}, we regularize the solution $c_i^\varepsilon$ with the aid of Steklov average and we obtain an integral inequality for the concentrations $c_i^\varepsilon$.
\begin{theorem}\label{teopos} Let $(c_{1}^{\varepsilon},\,c_{2}^{\varepsilon}, \,c_{3}^{\varepsilon})$ be the weak solution of problem \eqref{eqmicro}. Then, for a.e. $t\in(0,T)$ and for $i\in \{1,2,3\}$, we have
\begin{equation}\label{pos1}
c_i^\varepsilon(\cdot,x)\geq 0\,\text{ a.e. in }\Omega^{\ast}_{\varepsilon}.
\end{equation}
\end{theorem}
\begin{proof}
We start by proving that, for $i\in \{1,2\}$, for a.e. $t\in(0,T)$ we have
\begin{equation}\label{pos3}
\begin{array}{l}
\dfrac{1}{2}\Vert(c_i^\varepsilon)_{-}(t)\Vert^2_{L^2(\Omega^{\ast}_{\varepsilon})}+\alpha\left\Vert \nabla(c_{i}%
^{\varepsilon})_{-}\right\Vert ^{2}_{L^2(0,t;L^2(\Omega^{\ast}_{\varepsilon}))} \\
\\
\leq \displaystyle\int_0^{t}\int_{\Omega^{\ast}_{\varepsilon}}F_{i}^{\varepsilon}(x,c_{1}^{\varepsilon},c_{2}^{\varepsilon
}, c_{3}^{\varepsilon})(c_{i}^{\varepsilon
})_{-}\mathrm{d}x\, \mathrm{d}s+ \displaystyle \frac{1}{\varepsilon} \int_0^{t}\int_{\Gamma^{\varepsilon}%
}G_{i}(c_{1}^{\varepsilon},c_{2}^{\varepsilon
}, c_{3}^{\varepsilon})(c_{i}^{\varepsilon})_{-}\mathrm{d}\sigma_{x}\,\mathrm{d}s,
\end{array}
\end{equation}
while for $i=3$ we have
\begin{equation}\label{pos3-3}
\begin{array}{l}
\dfrac{1}{2}\Vert(c_3^\varepsilon)_{-}(t)\Vert^2_{L^2(\Omega^{\ast}_{\varepsilon})}+\alpha\left\Vert \nabla(c_{3}%
^{\varepsilon})_{-}\right\Vert ^{2}_{L^2(0,t;L^2(\Omega^{\ast}_{\varepsilon}))} \\
\\
\leq \displaystyle\int_0^{t}\int_{\Omega^{\ast}_{\varepsilon}}F_{3}^{\varepsilon}(x,c_{1}^{\varepsilon},c_{2}^{\varepsilon
}, c_{3}^{\varepsilon})(c_{3}^{\varepsilon
})_{-}\mathrm{d}x\, \mathrm{d}s+ \displaystyle \varepsilon \int_0^{t}\int_{\Gamma^{\varepsilon}%
}G_{3}^{\varepsilon}(x,c_{1}^{\varepsilon},c_{2}^{\varepsilon
}, c_{3}^{\varepsilon})(c_{3}^{\varepsilon})_{-}\mathrm{d}\sigma_{x}\,\mathrm{d}s.
\end{array}
\end{equation}
From the first equation in \eqref{weakconcentration}, after integration
in time from $t$ to $t + h$, by integration by parts in the time derivative term and by multiplying the whole equation by $\dfrac{1}{h}$, we get for all $t, h\in (0,T)$ with $(t+h)\in (0,T)$ and for all $v\in H_{\partial \Omega}^1(\Omega^{\ast}_{\varepsilon})$
\begin{equation}\label{pos1*}
\begin{array}{l}
\dfrac{1}{h}\displaystyle\int_{\Omega^{\ast}_{\varepsilon}}\left(c_i^\varepsilon(t+h)-c_i^\varepsilon(t)\right)v\,\mathrm{d}x +\dfrac{1}{h}\int_t^{t+h}\int_{\Omega^{\ast}_{\varepsilon}} D_i^\varepsilon\nabla c_{i}^{\varepsilon} \cdot\nabla v\,\mathrm{d}x \,\mathrm{d}s\\
\\
=\displaystyle\dfrac{1}{h}\int_t^{t+h}\int_{\Omega^{\ast}_{\varepsilon}}F_{i}^{\varepsilon}(x,c_{1}^{\varepsilon},c_{2}^{\varepsilon
}, c_{3}^{\varepsilon})
v\,\mathrm{d}x\, \mathrm{d}s+\dfrac{1}{h \varepsilon}\int_t^{t+h}\int_{\Gamma^{\varepsilon}%
}G_{i}(c_{1}^{\varepsilon},c_{2}^{\varepsilon
}, c_{3}^{\varepsilon})v\,\mathrm{d}\sigma_{x}\,\mathrm{d}s,\,\,\,\,i\in \{1,2\}.
\end{array}
\end{equation}
Using the Steklov average, defined for a function $u\in L^2(0,T;L^2(D))$ with $D\subset \mathbb{R}^n$ open, $t, h\in (0, T )$ by
\begin{equation}\label{Stek}
[u]_h(t):=\left\{
\begin{array}{ll}
\dfrac{1}{h}\displaystyle\int_t^{t+h} u(s,\cdot)\,\mathrm{d}s & \text{for } t\in (0,T-h],\\
\\
0 & \text{for } t>T-h,
\end{array}
\right.
\end{equation}
we obtain
\begin{equation}\label{pos2}
\begin{array}{l}
\displaystyle\int_{\Omega^{\ast}_{\varepsilon}}\partial_t[c_i^\varepsilon]_{h} v \, \mathrm{d}x +\int_{\Omega^{\ast}_{\varepsilon}}D_i^\varepsilon \nabla [c_{i}^{\varepsilon}]_h
\cdot\nabla v\mathrm{d}x \\
\\
=\displaystyle\int_{\Omega^{\ast}_{\varepsilon}}[F_{i}^{\varepsilon}(x,c_{1}^{\varepsilon},c_{2}^{\varepsilon
},c_{3}^{\varepsilon})]_h
v\,\mathrm{d}x+\displaystyle \frac{1}{\varepsilon} \int_{\Gamma^{\varepsilon}%
}[G_{i}(c_{1}^{\varepsilon},c_{2}^{\varepsilon
}, c_{3}^{\varepsilon})]_h v\,\mathrm{d}\sigma_{x},\,\,\,i\in \{1,2\}.
\end{array}
\end{equation}
We have $\partial_t [c_i^\varepsilon]_h\in L^2(0,T; L^2(\Omega^{\ast}_{\varepsilon}))$ for all $0 < h < \delta$ and therefore $\partial_t ([c_i^\varepsilon]_h)_{-}\in L^2(0,T-\delta; L^2(\Omega^{\ast}_{\varepsilon}))$. For the function $[c_i^\varepsilon]_h$ it holds that
$$\langle\partial_t [c_i^\varepsilon]_h(t), ([c_i^\varepsilon]_h(t))_{-}\rangle_{\Omega^{\ast}_{\varepsilon}}=\dfrac{1}{2}\dfrac{d}{dt}\Vert([c_i^\varepsilon]_h)_{-}\Vert^2_{L^2(\Omega^{\ast}_{\varepsilon})}$$
and we obtain by integration from $0$ to $t \in (0, T )$ and testing the first equation in \eqref{weakconcentration} with $v=([c_i^\varepsilon
]_h)_{-}$, due to the definition of the cut off function $(c_i^\varepsilon)_{-}$ and assumption $(\mathbf{H}_1)_{2}$, for small $h$ and $i\in \{1,2\}$
\begin{equation}\label{pos4}
\begin{array}{l}
\dfrac{1}{2}\Vert([c_i^\varepsilon]_h)_{-}(t)\Vert^2_{L^2(\Omega^{\ast}_{\varepsilon})}-\dfrac{1}{2}\Vert([c_i^\varepsilon]_h)_{-}(0)\Vert^2_{L^2(\Omega^{\ast}_{\varepsilon})}+\alpha \left\Vert \nabla([c_{i}%
^{\varepsilon}]_h)_{-}\right\Vert ^{2}_{L^2(0,t;L^2(\Omega^{\ast}_{\varepsilon}))}\\\\
\leq\displaystyle\int_0^{t}\int_{\Omega^{\ast}_{\varepsilon}}[F_{i}^{\varepsilon}(x,c_{1}^{\varepsilon},c_{2}^{\varepsilon
}, c_{3}^{\varepsilon})]_h([c_{i}^{\varepsilon
}]_h)_{-}\mathrm{d}x\,\mathrm{d}s+\displaystyle \frac{1}{\varepsilon}\int_0^{t}\int_{\Gamma^{\varepsilon}%
}[G_{i}(c_{1}^{\varepsilon},c_{2}^{\varepsilon
}, c_{3}^{\varepsilon})]_h([c_{i}^{\varepsilon}]_h)_{-}\mathrm{d}\sigma_{x}\, \mathrm{d}s.
\end{array}
\end{equation}
Using the properties of the Steklov average, we can pass to the limit as $h\rightarrow 0$ in \eqref{pos4} and by assumption $(\mathbf{H}_7)$ we use that the initial concentration fields are nonnegative, i.e. ($c_i^\varepsilon (0))_{-}=0$, to obtain \eqref{pos3}. In a similar manner one can obtain \eqref{pos3-3}.

Summing up equations \eqref{pos3} for $i\in \{1,2\}$ and \eqref{pos3-3}, we get
\begin{equation}
\begin{array}{l}
\dfrac{1}{2}\left\Vert (c_{1}^{\varepsilon}%
)_{-}(t)\right\Vert ^{2}_{L^2(\Omega^{\ast}_{\varepsilon})}+ \dfrac{1}{2}\left\Vert (c_{2}^{\varepsilon}%
)_{-}(t)\right\Vert ^{2}_{L^2(\Omega^{\ast}_{\varepsilon})}+ \dfrac{1}{2}\left\Vert (c_{3}^{\varepsilon}%
)_{-}(t)\right\Vert ^{2}_{L^2(\Omega^{\ast}_{\varepsilon})} +\\
\\ \alpha\left(\left\Vert \nabla(c_{1}%
^{\varepsilon})_{-}\right\Vert ^{2}_{L^2(0,t;L^2(\Omega^{\ast}_{\varepsilon}))}+\left\Vert \nabla(c_{2}^{\varepsilon}%
)_{-}\right\Vert ^{2}_{L^2(0,t;L^2(\Omega^{\ast}_{\varepsilon}))}+\left\Vert \nabla(c_{3}^{\varepsilon}%
)_{-}\right\Vert ^{2}_{L^2(0,t;L^2(\Omega^{\ast}_{\varepsilon}))}\right)\\
\\
\leq \displaystyle\int_0^{t}\int_{\Omega^{\ast}_{\varepsilon}}\left[F_{1}^{\varepsilon}(x,c_{1}^{\varepsilon},c_{2}^{\varepsilon
}, c_{3}^{\varepsilon})(c_{1}^{\varepsilon})_{-}%
+F_{2}^{\varepsilon}(x,c_{1}^{\varepsilon},c_{2}^{\varepsilon
}, c_{3}^{\varepsilon})(c_{2}^{\varepsilon})_{-}+F_{3}^{\varepsilon}(x,c_{1}^{\varepsilon},c_{2}^{\varepsilon
}, c_{3}^{\varepsilon})(c_{3}^{\varepsilon})_{-}\right]\mathrm{d}x\, \mathrm{d}s\\
\\+
\displaystyle \frac{1}{\varepsilon}\int_0^{t}\int_{\Gamma^{\varepsilon}}\left[G_{1}(c_{1}^{\varepsilon},c_{2}^{\varepsilon
}, c_{3}^{\varepsilon})(c_{1}^{\varepsilon})_{-}%
+G_{2}(c_{1}^{\varepsilon},c_{2}^{\varepsilon
},c_{3}^{\varepsilon})(c_{2}^{\varepsilon})_{-}\right]\mathrm{d}\sigma_{x}\, \mathrm{d}s \\
\\ +
\displaystyle \varepsilon\int_0^{t}\int_{\Gamma^{\varepsilon}} G_{3}^{\varepsilon}(x,c_{1}^{\varepsilon},c_{2}^{\varepsilon
}, c_{3}^{\varepsilon})(c_{3}^{\varepsilon})_{-}\mathrm{d}\sigma_{x}\, \mathrm{d}s.
\end{array}
\end{equation}
By taking into account  \eqref{formG}, \eqref{assF2} and \eqref{assG}, we obtain
\begin{equation*}
\begin{array}{l}
\dfrac{1}%
{2}\left(  \left\Vert (c_{1}^{\varepsilon})_{-}(t)\right\Vert ^{2}_{L^2(\Omega^{\ast}_{\varepsilon})}+\left\Vert
(c_{2}^{\varepsilon})_{-}(t)\right\Vert ^{2}_{L^2(\Omega^{\ast}_{\varepsilon})}+\left\Vert (c_{3}^{\varepsilon})_{-}(t)\right\Vert ^{2}_{L^2(\Omega^{\ast}_{\varepsilon})}\right) \\
\\
+\alpha\left(\left\Vert \nabla(c_{1}%
^{\varepsilon})_{-}\right\Vert ^{2}_{L^2(0,t;L^2(\Omega^{\ast}_{\varepsilon}))}+\left\Vert \nabla(c_{2}^{\varepsilon}%
)_{-}\right\Vert ^{2}_{L^2(0,t;L^2(\Omega^{\ast}_{\varepsilon}))}+\left\Vert \nabla(c_{3}%
^{\varepsilon})_{-}\right\Vert ^{2}_{L^2(0,t;L^2(\Omega^{\ast}_{\varepsilon}))}\right)\\
\\
\leq C_1\displaystyle\int_0^t\left(\Vert (c_1^\varepsilon)_{-}(s)\Vert^2_{L^2(\Omega^{\ast}_{\varepsilon})}+\Vert (c_2^\varepsilon)_{-}(s)\Vert^2_{L^2(\Omega^{\ast}_{\varepsilon})} +\Vert (c_3^\varepsilon)_{-}(s)\Vert^2_{L^2(\Omega^{\ast}_{\varepsilon})}\right) \, \mathrm{d}s\\
\\
+ \displaystyle C_2 \, \varepsilon
\int_0^t\left(\Vert (c_1^\varepsilon)_{-}(s)\Vert^2_{L^2(\Gamma^\varepsilon)}+\Vert (c_2^\varepsilon)_{-}(s)\Vert^2_{L^2(\Gamma^\varepsilon)} +\Vert (c_3^\varepsilon)_{-}(s)\Vert^2_{L^2(\Gamma^\varepsilon)}\right) \, \mathrm{d}s\\
\\
+\displaystyle \frac{1}{\varepsilon}\int_0^{t}\int_{\Gamma^{\varepsilon}}\left(c_{1}^{\varepsilon}-c_{2}^{\varepsilon
}\right)H\left(c_3^\varepsilon\right)\left[(c_2^\varepsilon)_{-}-(c_1^\varepsilon)_{-}\right]\mathrm{d}\sigma_{x}\, \mathrm{d}s.
\end{array}
\end{equation*}
Since the function $H$ is positive, by the definition of $(c_i^\varepsilon)_-$, it is easy to show that
\begin{equation}\label{negG}\displaystyle \frac{1}{\varepsilon}\int_0^{t}\int_{\Gamma^{\varepsilon}}\left(c_{1}^{\varepsilon}-c_{2}^{\varepsilon
}\right)H\left(c_3^\varepsilon\right)\left[(c_2^\varepsilon)_{-}-(c_1^\varepsilon)_{-}\right]\mathrm{d}\sigma_{x}\, \mathrm{d}s\leq 0,
\end{equation}
and, hence, we get
\begin{equation*}
\begin{array}{l}
\dfrac{1}%
{2}\left(  \left\Vert (c_{1}^{\varepsilon})_{-}(t)\right\Vert ^{2}_{L^2(\Omega^{\ast}_{\varepsilon})}+\left\Vert
(c_{2}^{\varepsilon})_{-}(t)\right\Vert ^{2}_{L^2(\Omega^{\ast}_{\varepsilon})}+\left\Vert (c_{3}^{\varepsilon})_{-}(t)\right\Vert ^{2}_{L^2(\Omega^{\ast}_{\varepsilon})}\right) \\
\\
+\alpha\left(\left\Vert \nabla(c_{1}%
^{\varepsilon})_{-}\right\Vert ^{2}_{L^2(0,t;L^2(\Omega^{\ast}_{\varepsilon}))}+\left\Vert \nabla(c_{2}^{\varepsilon}%
)_{-}\right\Vert ^{2}_{L^2(0,t;L^2(\Omega^{\ast}_{\varepsilon}))}+\left\Vert \nabla(c_{3}%
^{\varepsilon})_{-}\right\Vert ^{2}_{L^2(0,t;L^2(\Omega^{\ast}_{\varepsilon}))}\right)\\
\\
\leq C_1\displaystyle\int_0^t\left(\Vert (c_1^\varepsilon)_{-}(s)\Vert^2_{L^2(\Omega^{\ast}_{\varepsilon})}+\Vert (c_2^\varepsilon)_{-}(s)\Vert^2_{L^2(\Omega^{\ast}_{\varepsilon})} +\Vert (c_3^\varepsilon)_{-}(s)\Vert^2_{L^2(\Omega^{\ast}_{\varepsilon})}\right) \, \mathrm{d}s\\
\\
+\displaystyle C_2 \, \varepsilon
\int_0^t\left(\Vert (c_1^\varepsilon)_{-}(s)\Vert^2_{L^2(\Gamma^\varepsilon)}+\Vert (c_2^\varepsilon)_{-}(s)\Vert^2_{L^2(\Gamma^\varepsilon)} +\Vert (c_3^\varepsilon)_{-}(s)\Vert^2_{L^2(\Gamma^\varepsilon)}\right) \, \mathrm{d}s.
\end{array}
\end{equation*}
Using Gronwall's inequality, we are led, as in \cite{Gahn}, to \eqref{pos1}.
\end{proof}

\bigskip

Following the same argument as in \cite{Gahn3}, we prove that $c_{i}^{\varepsilon},\, i\in \{1,2, 3\}$, are essentially bounded.
\begin{theorem}\label{teobound} Let $(c_{1}^{\varepsilon},\,c_{2}^{\varepsilon},\,c_{3}^{\varepsilon})$ be the weak solution of problem \eqref{eqmicro}. Then, for a.e. $t\in(0,T)$ and for $i\in \{1,2, 3\}$, there exists a constant $C>0$ independent of $\varepsilon$ such that
\begin{equation}\label{b1}
\left\Vert c_i^\varepsilon \right\Vert_{L^\infty((0,T)\times \Omega^{\ast}_{\varepsilon})} \leq C.
\end{equation}
\end{theorem}
\begin{proof}
Set $\left( \cdot\right)_+  =\max\left\{\cdot,0\right\}$. Let $A$ and $\Lambda$ be as in assumptions $(\mathbf{H}_6)$ and $(\mathbf{H}_7)$  and let us consider, for $i\in \{1,2, 3\}$, $v=e^{-tA}(c_i^\varepsilon \, e^{-tA}-\Lambda)_+$ as test function in \eqref{weakconcentration}. Then, we obtain, for $i\in \{1,2\}$,
\begin{equation*}
\begin{array}{l}
\langle\partial_{t}c_{i}^{\varepsilon}, e^{-tA}(c_i^\varepsilon \, e^{-tA}-\Lambda)_+\rangle_{\Omega^{\ast}_{\varepsilon}}+\displaystyle\int_{\Omega^{\ast}_{\varepsilon}} D_i^\varepsilon \nabla c_{i}^{\varepsilon} \cdot\nabla e^{-tA}(c_i^\varepsilon \, e^{-tA}-\Lambda)_+ \, \mathrm{d}x \\
\\
= \displaystyle \int_{\Omega^{\ast}_{\varepsilon}}F_{i}^{\varepsilon}(x,c_{1}^{\varepsilon},c_{2}^{\varepsilon
},c_{3}^{\varepsilon}) e^{-tA}(c_i^\varepsilon \, e^{-tA}-\Lambda)_+\,\mathrm{d}x +\displaystyle \frac{1}{\varepsilon} \int_{\Gamma^{\varepsilon}}G_{i}(c_{1}%
^{\varepsilon},c_{2}^{\varepsilon}, c_{3}^{\varepsilon}) e^{-tA}(c_i^\varepsilon \, e^{-tA}-\Lambda)_+\,\mathrm{d}\sigma
_{x}
\end{array}
\end{equation*}
and, for $i=3$,
\begin{equation*}
\begin{array}{l}
\langle\partial_{t}c_{3}^{\varepsilon}, e^{-tA}(c_3^\varepsilon \, e^{-tA}-\Lambda)_+\rangle_{\Omega^{\ast}_{\varepsilon}}+\displaystyle\int_{\Omega^{\ast}_{\varepsilon}} D_3^\varepsilon \nabla c_{3}^{\varepsilon} \cdot\nabla e^{-tA}(c_3^\varepsilon \, e^{-tA}-\Lambda)_+ \, \mathrm{d}x \\
\\= \displaystyle \int_{\Omega^{\ast}_{\varepsilon}}F_{3}^{\varepsilon}(x,c_{1}^{\varepsilon},c_{2}^{\varepsilon
},c_{3}^{\varepsilon}) e^{-tA}(c_{3}^\varepsilon \, e^{-tA}-\Lambda)_+\,\mathrm{d}x+\displaystyle \varepsilon \int_{\Gamma^{\varepsilon}}G_{3}^{\varepsilon}(x,c_{1}%
^{\varepsilon},c_{2}^{\varepsilon}, c_{3}^{\varepsilon}) e^{-tA}(c_3^\varepsilon \, e^{-tA}-\Lambda)_+\,\mathrm{d}\sigma
_{x}.
\end{array}
\end{equation*}
Using assumption $(\mathbf{H}_6)$  and \eqref{formG}, for $i\in \{1,2\}$, we have
\begin{equation} \label{est12}
\begin{array}{l}
\langle\partial_{t}c_{i}^{\varepsilon}, e^{-tA}(c_i^\varepsilon \, e^{-tA}-\Lambda)_+\rangle_{\Omega^{\ast}_{\varepsilon}}+\displaystyle\int_{\Omega^{\ast}_{\varepsilon}} D_i^\varepsilon \nabla c_{i}^{\varepsilon} \cdot\nabla e^{-tA}(c_i^\varepsilon \, e^{-tA}-\Lambda)_+ \, \mathrm{d}x \\
\\
\leq \displaystyle \int_{\Omega^{\ast}_{\varepsilon}} A e^{-tA} c_i^\varepsilon (c_i^\varepsilon \, e^{-tA}-\Lambda)_+\,\mathrm{d}x
+\displaystyle \frac{1}{\varepsilon} \int_{\Gamma^{\varepsilon}}(-1)^i e^{-tA}(c^\varepsilon_{1}-c^\varepsilon_{2})H\left(c_{3}\right)\left(c_i^\varepsilon e^{-tA} -\Lambda\right)_+\,\mathrm{d}\sigma_{x}
\end{array}
\end{equation}
and, for $i=3$,
\begin{equation} \label{est3}
\begin{array}{l}
\langle\partial_{t}c_{3}^{\varepsilon}, e^{-tA}(c_3^\varepsilon \, e^{-tA}-\Lambda)_+\rangle_{\Omega^{\ast}_{\varepsilon}}+\displaystyle\int_{\Omega^{\ast}_{\varepsilon}} D_3^\varepsilon \nabla c_{3}^{\varepsilon} \cdot\nabla e^{-tA}(c_3^\varepsilon \, e^{-tA}-\Lambda)_+ \, \mathrm{d}x \\
\\
\leq \displaystyle \int_{\Omega^{\ast}_{\varepsilon}} A e^{-tA} c_3^\varepsilon (c_3^\varepsilon \, e^{-tA}-\Lambda)_+\,\mathrm{d}x+\displaystyle \varepsilon \int_{\Gamma^{\varepsilon}} A e^{-tA} c_3^\varepsilon (c_3^\varepsilon \, e^{-tA}-\Lambda)_+\,\mathrm{d}\sigma
_{x}\\
\\= \displaystyle \int_{\Omega^{\ast}_{\varepsilon}} A e^{-tA} c_3^\varepsilon (c_3^\varepsilon \, e^{-tA}-\Lambda)_+\,\mathrm{d}x+ \displaystyle A \varepsilon  \left\Vert (c_3^\varepsilon \, e^{-tA}-\Lambda)_+\right \Vert ^{2}_{L^2(\Gamma^{\varepsilon})} \,\mathrm{d}\sigma
_{x}+ \displaystyle A \varepsilon  \int_{\Gamma^{\varepsilon}} \Lambda (c_3^\varepsilon \, e^{-tA}-\Lambda)_+\,\mathrm{d}\sigma
_{x}.
\end{array}
\end{equation}
On the other hand, for $i\in \{1,2, 3\}$, one has\\
\begin{equation}\label{der}
\begin{array}{l}
\langle\partial_{t}c_{i}^{\varepsilon}, e^{-tA}(c_i^\varepsilon \, e^{-tA}-\Lambda)_+\rangle_{\Omega^{\ast}_{\varepsilon}}\\
 \\=\langle\partial_{t}(c_i^\varepsilon \, e^{-tA}-\Lambda), (c_i^\varepsilon \, e^{-tA}-\Lambda)_+\rangle_{\Omega^{\ast}_{\varepsilon}} + \langle A e^{-tA}c_{i}^{\varepsilon}, (c_i^\varepsilon \, e^{-tA}-\Lambda)_+\rangle_{\Omega^{\ast}_{\varepsilon}}.
\end{array}
\end{equation}
Therefore, summing up \eqref{est12} for $i\in \{1,2\}$, by assumption $(\mathbf{H}_1)_{2}$ we obtain
\begin{equation*}
\begin{array}{l}
\displaystyle \frac{\mathrm{d}}{\mathrm{d}t}\left\Vert (c_1^\varepsilon \, e^{-tA}-\Lambda)_+ \right\Vert ^{2}_{L^2(\Omega^{\ast}_{\varepsilon})}+ \displaystyle \frac{\mathrm{d}}{\mathrm{d}t} \left\Vert (c_2^\varepsilon \, e^{-tA}-\Lambda)_+\right\Vert ^{2}_{L^2(\Omega^{\ast}_{\varepsilon})}\\
\\
+\alpha\left(\left\Vert \nabla (c_1^\varepsilon \, e^{-tA}-\Lambda)_+ \right\Vert ^{2}_{L^2(\Omega^{\ast}_{\varepsilon})}
+\left\Vert \nabla (c_2^\varepsilon \, e^{-tA}-\Lambda)_+ \right\Vert ^{2}_{L^2(\Omega^{\ast}_{\varepsilon})}\right)\\
\\
 \leq \displaystyle \frac{1}{\varepsilon} \int_{\Gamma^{\varepsilon}}\left[(c^\varepsilon_{1}e^{-tA}-\Lambda)-(c^\varepsilon_{2}e^{-tA}-\Lambda)\right]H\left(c_{3}\right)\left[(c_2^\varepsilon e^{-tA} -\Lambda)_+ - (c_1^\varepsilon e^{-tA} -\Lambda)_+ \right] \,\mathrm{d}\sigma_{x}.
\end{array}
\end{equation*}
By arguing as to prove \eqref{negG}, we get that the second member of the previous inequality is non positive and then
\begin{equation*}
\displaystyle \frac{\mathrm{d}}{\mathrm{d}t}\left\Vert (c_1^\varepsilon \, e^{-tA}-\Lambda)_+ \right\Vert ^{2}_{L^2(\Omega^{\ast}_{\varepsilon})}+ \displaystyle \frac{\mathrm{d}}{\mathrm{d}t} \left\Vert (c_2^\varepsilon \, e^{-tA}-\Lambda)_+\right\Vert ^{2}_{L^2(\Omega^{\ast}_{\varepsilon})}\leq 0.
\end{equation*}
Integrating in  time and taking into account that the initial data are bounded, we get the essential boundedness of $c_1^\varepsilon$ and $c_2^\varepsilon$.

From \eqref{est3}, taking into account \eqref{der} and assumption $(\mathbf{H}_1)_{2}$, we get
\begin{equation*}
\begin{array}{l}
\displaystyle \frac{\mathrm{d}}{\mathrm{d}t}\left\Vert (c_3^\varepsilon \, e^{-tA}-\Lambda)_+ \right\Vert ^{2}_{L^2(\Omega^{\ast}_{\varepsilon})}+\alpha \left\Vert \nabla (c_3^\varepsilon \, e^{-tA}-\Lambda)_+ \right\Vert ^{2}_{L^2(\Omega^{\ast}_{\varepsilon})} \\
\\ \leq C  \Big ( \displaystyle \varepsilon  \left\Vert (c_3^\varepsilon \, e^{-tA}-\Lambda)_+\right \Vert ^{2}_{L^2(\Gamma^{\varepsilon})} + \displaystyle \varepsilon  \int_{\Gamma^{\varepsilon}} \Lambda  (c_3^\varepsilon \, e^{-tA}-\Lambda)_+\,\mathrm{d}\sigma
_{x} \Big ).
\end{array}
\end{equation*}
Now, exactly as in \cite[Proposition 3.38]{Gahn3}, we are led to the essential boundedness of $c_3^\varepsilon$.
\end{proof}

\subsection{Estimates for the microscopic model}\label{secest}
Our goal is to obtain the effective behavior of the solutions of the microscopic system \eqref{eqmicro}. To this aim, we need to pass to the limit, with $\varepsilon \rightarrow 0$, in the variational formulation of the microscopic model, by using compactness results with respect to suitable topologies. Thus, we need to prove {\it a priori} estimates for our solution.
\begin{lemma}\label{lemest} For the solution $(c_{1}^{\varepsilon},\,c_{2}^{\varepsilon}, \,c_{3}^{\varepsilon})$ of problem \eqref{eqmicro}, there exists a positive constant $C>0$, independent of $\varepsilon$, such that the following estimates hold true:
\begin{equation}
\Vert c_i^{\varepsilon}\Vert_{L^{\infty}(0,T;L^{2}(\Omega^{\ast}_{\varepsilon}))}\leq C\,\,\,\,i\in \{1,2,3\},
\label{estimconcLinf}
\end{equation}
\begin{equation}
\Vert c_i^{\varepsilon}\Vert_{L^{2}(0,T;H^{1}(\Omega^{\ast}_{\varepsilon}))}\leq C\,\,\,\,i\in \{1,2,3\},
\label{estimconcH1}
\end{equation}
\begin{equation}
\Vert c_1^{\varepsilon}-c_2^{\varepsilon}\Vert_{L^{2}(0,T;L^{2}(\Gamma^{\varepsilon}))}\leq C \sqrt{\varepsilon},
\label{estimdifconc}
\end{equation}
\begin{equation}
\sqrt{\varepsilon}\Vert c_3^{\varepsilon}\Vert_{L^{2}(0,T;L^{2}(\Gamma^{\varepsilon}))}\leq C,
\label{estimconcb}
\end{equation}
\begin{equation}
\Vert \partial_t c_i^{\varepsilon}\Vert_{L^{2}(0,T;(H_0^{1}(\Omega^{\ast}_{\varepsilon}))')}\leq C\,\,\,\,i\in \{1,2\},
\label{estimconcH1'}
\end{equation}
\begin{equation}
\Vert \partial_t c_3^{\varepsilon}\Vert_{L^{2}(0,T;(H_{\partial \Omega}^{1}(\Omega^{\ast}_{\varepsilon}))')}\leq C.
\label{estimconc3H1'}
\end{equation}
\end{lemma}
\begin{proof}
Let us prove \eqref{estimconcLinf}. To this end, we test the first equation in \eqref{weakconcentration} with $c_{i}^{\varepsilon}$, for $i\in \{1,2\}$, and the second one with $c_{3}^{\varepsilon}$, respectively. We get
\begin{equation*}
\begin{array}{l}
\displaystyle\frac{1}{2}\frac{\mathrm{d}}{\mathrm{d}t}\Vert c_{i}^{\varepsilon}\Vert
_{L^{2}(\Omega^{\ast}_{\varepsilon})}^{2}+\int_{\Omega^{\ast}_{\varepsilon}}D_i^\varepsilon \nabla
c_{i}^{\varepsilon}\cdot\nabla
c_{i}^{\varepsilon}\mathrm{d}x  =\displaystyle\int_{\Omega^{\ast}_{\varepsilon}}F_{i}%
^{\varepsilon}(x,c_{1}^{\varepsilon},c_{2}^{\varepsilon
},c_{3}^{\varepsilon})c_{i}^{\varepsilon}\mathrm{d}x+\displaystyle \frac{1}{\varepsilon}\int_{\Gamma
^{\varepsilon}}G_{i}(c_{1}^{\varepsilon},c_{2}^{\varepsilon
}, c_{3}^{\varepsilon})c_{i}^{\varepsilon}\mathrm{d}\sigma_{x}
\end{array}
\end{equation*}
and
\begin{equation*}
\begin{array}{l}
\displaystyle\frac{1}{2}\frac{\mathrm{d}}{\mathrm{d}t}\Vert c_{3}^{\varepsilon}\Vert
_{L^{2}(\Omega^{\ast}_{\varepsilon})}^{2}+\int_{\Omega^{\ast}_{\varepsilon}}D_3^\varepsilon \nabla
c_{3}^{\varepsilon}\cdot\nabla
c_{3}^{\varepsilon}\mathrm{d}x  =\displaystyle\int_{\Omega^{\ast}_{\varepsilon}}F_{3}%
^{\varepsilon}(x,c_{1}^{\varepsilon},c_{2}^{\varepsilon
},c_{3}^{\varepsilon})c_{3}^{\varepsilon}\mathrm{d}x+\displaystyle \varepsilon\int_{\Gamma
^{\varepsilon}}G_{3}^{\varepsilon}(x,c_{1}^{\varepsilon},c_{2}^{\varepsilon
}, c_{3}^{\varepsilon})c_{3}^{\varepsilon}\mathrm{d}\sigma_{x}.
\end{array}
\end{equation*}
Then, by  \eqref{formG} and assumption $(\mathbf{H}_1)_{2}$, we get
\begin{equation*}
\begin{array}{l}
\displaystyle\frac{1}{2}\left (\frac{\mathrm{d}}{\mathrm{d}t}\Vert c_{1}^{\varepsilon}\Vert
_{L^{2}(\Omega^{\ast}_{\varepsilon})}^{2}+\displaystyle \frac{\mathrm{d}}{\mathrm{d}t}\Vert c_{1}^{\varepsilon}\Vert
_{L^{2}(\Omega^{\ast}_{\varepsilon})}^{2}+\displaystyle \frac{\mathrm{d}}{\mathrm{d}t}\Vert c_{3}^{\varepsilon}\Vert
_{L^{2}(\Omega^{\ast}_{\varepsilon})}^{2} \right )\\
\\
+\alpha\left(\Vert\nabla c_{1}^{\varepsilon}\Vert
_{L^{2}(\Omega^{\ast}_{\varepsilon})}^{2}+\Vert\nabla c_{2}^{\varepsilon}\Vert
_{L^{2}(\Omega^{\ast}_{\varepsilon})}^{2}+\Vert\nabla c_{3}^{\varepsilon}\Vert
_{L^{2}(\Omega^{\ast}_{\varepsilon})}^{2}\right)\\
\\
\leq \displaystyle\int_{\Omega^{\ast}_{\varepsilon}}F_{1}^{\varepsilon}(x,c_{1}^{\varepsilon},c_{2}^{\varepsilon
},c_{3}^{\varepsilon})c_{1}%
^{\varepsilon}\mathrm{d}x+\displaystyle\int_{\Omega^{\ast}_{\varepsilon}}F_{2}^{\varepsilon}(x,c_{1}^{\varepsilon},c_{2}^{\varepsilon
},c_{3}^{\varepsilon})c_{2}%
^{\varepsilon}\mathrm{d}x+\displaystyle\int_{\Omega^{\ast}_{\varepsilon}}F_{3}^{\varepsilon}(x,c_{1}^{\varepsilon},c_{2}^{\varepsilon
},c_{3}^{\varepsilon})c_{3}%
^{\varepsilon}\mathrm{d}x\\
\\
+\displaystyle \varepsilon\int_{\Gamma^{\varepsilon}}%
G_{3}^{\varepsilon}(x,c_{1}^{\varepsilon},c_{2}^{\varepsilon
},c_{3}^{\varepsilon})c_{3}^{\varepsilon}\mathrm{d}\sigma_{x} -\displaystyle \dfrac{1}{\varepsilon}\int_{\Gamma^{\varepsilon}}%
\left (c_{1}^{\varepsilon}-c_{2}^{\varepsilon}\right )^2 H \left (c_{3}\right )\mathrm{d}\sigma_{x}.
\end{array}
\end{equation*}
Since $H$ is nonnegative, by using the growth conditions \eqref{F2} and \eqref{assG1} for $F_i^\varepsilon$ and $G_3^\varepsilon$, respectively, we have
\begin{equation}\label{estc4}
\begin{array}{l}
\displaystyle\frac{1}{2}\left (\frac{\mathrm{d}}{\mathrm{d}t}\Vert c_{1}^{\varepsilon}\Vert
_{L^{2}(\Omega^{\ast}_{\varepsilon})}^{2}+\displaystyle \frac{\mathrm{d}}{\mathrm{d}t}\Vert c_{1}^{\varepsilon}\Vert
_{L^{2}(\Omega^{\ast}_{\varepsilon})}^{2}+\displaystyle \frac{\mathrm{d}}{\mathrm{d}t}\Vert c_{3}^{\varepsilon}\Vert
_{L^{2}(\Omega^{\ast}_{\varepsilon})}^{2} \right )\\
\\
+\alpha\left(\Vert\nabla c_{1}^{\varepsilon}\Vert
_{L^{2}(\Omega^{\ast}_{\varepsilon})}^{2}+\Vert\nabla c_{2}^{\varepsilon}\Vert
_{L^{2}(\Omega^{\ast}_{\varepsilon})}^{2}+\Vert\nabla c_{3}^{\varepsilon}\Vert
_{L^{2}(\Omega^{\ast}_{\varepsilon})}^{2}\right)\\
\\
 \leq C \displaystyle\int
_{\Omega^{\ast}_{\varepsilon}}(1+|c_{1}^\varepsilon%
|+|c_{2}^\varepsilon|+|c_{3}^{\varepsilon}| )(|c_{1}^{\varepsilon}|+|c_{2}^{\varepsilon}|+|c_{3}^{\varepsilon}|)\mathrm{d}%
x+C\varepsilon\int_{\Gamma^{\varepsilon}}(1+|c_{1}^\varepsilon
|+|c_{2}^\varepsilon|+|c_{3}^{\varepsilon}| )|c_{3}^{\varepsilon}|\mathrm{d}\sigma_{x}.
\end{array}
\end{equation}
Using the inequality $2ab\leq a^2+b^2$, the fact that $|\Gamma^\varepsilon|\leq C \varepsilon^{-1}$, $|\Omega^{\ast}_{\varepsilon}|\leq C$ and $\varepsilon<1$, as well as the modified trace inequality from Lemma \ref{lemtrace}$ (iii)$, as in \cite{Gahn} or \cite{Gahn3}, we are led to
\begin{equation}\label{est7}
\begin{array}{ll}
\displaystyle \frac{\mathrm{d}}{\mathrm{d}t}\left(  \Vert c_{1}^{\varepsilon}\Vert
_{L^{2}(\Omega^{\ast}_{\varepsilon})}^{2}+\Vert c_{2}^{\varepsilon}\Vert
_{L^{2}(\Omega^{\ast}_{\varepsilon})}^{2}+ \Vert c_{3}^{\varepsilon}\Vert
_{L^{2}(\Omega^{\ast}_{\varepsilon})}^{2}\right) + \alpha\left(\Vert\nabla c_{1}^{\varepsilon}\Vert_{L^{2}(\Omega^{\ast}_{\varepsilon})}^{2}+\Vert\nabla c_{2}^{\varepsilon}%
\Vert_{L^{2}(\Omega^{\ast}_{\varepsilon})}^{2} + \Vert\nabla c_{3}^{\varepsilon}\Vert_{L^{2}(\Omega^{\ast}_{\varepsilon})}^{2} \right) \\
\\
\leq C_{1}+C_{2}\left(  \Vert
c_{1}^{\varepsilon}\Vert_{L^{2}(\Omega^{\ast}_{\varepsilon})}^{2}+\Vert
c_{2}^{\varepsilon}\Vert_{L^{2}(\Omega^{\ast}_{\varepsilon})}^{2} + \Vert
c_{3}^{\varepsilon}\Vert_{L^{2}(\Omega^{\ast}_{\varepsilon})}^{2}\right) .
\end{array}
\end{equation}
Integrating with respect to time and using Gronwall's inequality, we obtain the pointwise boundedness of $\Vert c_i^\varepsilon\Vert_{L^2(\Omega^{\ast}_{\varepsilon})}$, $i\in\{1,2,3\}$ in $(0,T)$, i.e. \eqref{estimconcLinf}. Then, \eqref{estimconcLinf} and estimate \eqref{est7} imply inequality \eqref{estimconcH1}.

In order to prove estimate \eqref{estimdifconc}, let us test the first equation in the variational formulation \eqref{weakconcentration}, written for $i=1$ and $i=2$ with $c_1^\varepsilon$ and $c_2^\varepsilon$, respectively. Summing up these two equations, moving the boundary terms in the left-hand side and taking into account \eqref{formG}, \eqref{pos1} and \eqref{b1}, we are led to \eqref{estimdifconc}.

In order to prove the estimates for the time derivative of the concentration fields, i.e. inequality \eqref{estimconcH1'}, we test the first equation in \eqref{weakconcentration} with $v\in H_0^1(\Omega^{\ast}_{\varepsilon})$, such that $\Vert v\Vert_{H_0^{1}(\Omega^{\ast}_{\varepsilon})}\leq 1$. Using similar arguments as before and the estimates \eqref{estimconcH1}, we get \eqref{estimconcH1'} for $i\in {1,2}$.
\\
In fact, for $c_3^{\varepsilon}$ we get a slightly better estimate, as in \cite{Gahn}. It is enough to test the second equation in  \eqref{weakconcentration} with $v\in H_{\partial \Omega}^{1}(\Omega^{\ast}_{\varepsilon})$ such that $\Vert v\Vert_{H_{\partial \Omega}^{1}(\Omega^{\ast}_{\varepsilon})}\leq 1$ and using similar arguments as before.
Estimate \eqref{estimconcb} follows directly from \eqref{estimconcH1} and Lemma \ref{lemtrace}$(ii)$.
\end{proof}

\bigskip

\section{Homogenization results by the periodic unfolding method}\label{Sec4}

In this section, we are interested in obtaining the effective behavior, as $\varepsilon\rightarrow0$,
of the solution $(c_{1}^{\varepsilon}, c_{2}^{\varepsilon}, c_{3}^{\varepsilon})$ of problem \eqref{weakconcentration}-\eqref{weakinitial}. To this aim, we shall use the {\it a priori} estimates given in Lemma \ref{lemest} to derive convergence results for the sequences $(c_{1}^{\varepsilon}, c_{2}^{\varepsilon}, c_{3}^{\varepsilon})$. In order to pass to the limit in the nonlinear terms in the variational formulation of \eqref{eqmicro}, we need to establish strong convergence results. For  using classical compactness results, we shall extend the functions $(c_{1}^{\varepsilon}, c_{2}^{\varepsilon}, c_{3}^{\varepsilon})$ to the whole domain $\Omega$ and we shall use unfolding operators, which transform functions on varying domains to functions on fixed domains (see, for instance, \cite{Cio-Dam-Don-Gri-Zaki}, \cite{Cio-Dam-Gri1}, \cite{Cioranescu-Donato-Zaki}, \cite {Donato}, \cite{Amar1}, and \cite{Amar2}). Through the unfolding method, which is more or less equivalent to the two-scale
convergence (see \cite{Allaire4} and \cite{N}), we can handle easier the nonlinearities on  $\Gamma^\varepsilon$.

\subsection{The time-dependent unfolding operator}\label{sub2.3}

In this subsection, we start by briefly recalling the definition and the main properties of the unfolding operator $\mathcal{T}^\ast_{\varepsilon }$ introduced, for time-dependent functions, in \cite{Donato-Yang*} (see also \cite{Donato-Yang}). Since time is a parameter, the results in \cite{Donato-Yang*} are direct generalizations of the corresponding ones from \cite{Cio-Dam-Don-Gri-Zaki}. For a more general setting of unfolding operators with time, we refer to \cite{Amar1} (see also \cite{Amar2}). For the unfolding operator defined in fixed domains, we refer the reader to \cite{Cio-Dam-Gri1}.
In the sequel, for  $z\in \mathbb{R}^{n}$, we use $\left[ z\right] _{Y}$ to
denote its integer part $k$,  such that $z-\left[ z\right] _{Y}\in Y$,
and we set
\begin{equation*}
\left\{ z\right\} _{Y}=z-\left[ z\right] _{Y}\in Y\text{ \ \ \ \ in } \mathbb{R}^{n}.
\end{equation*}%
Then,  for $x\in \mathbb{R}^n$, one has%
\begin{equation*}
x=\varepsilon \left( \left[ \frac{x}{\varepsilon }\right] _{Y}+\left\{ \frac{%
x}{\varepsilon }\right\} _{Y}\right).\end{equation*}%
In order to define the periodic unfolding operators, let us introduce the following sets as in \cite{Cio-Dam-Don-Gri-Zaki, Cio-Dam-Gri1}. Let
\begin{equation}\label{Omegahat}
\widehat{\Omega}_\varepsilon=\text{interior} \left\{\bigcup_{{\bf k} \in K_\varepsilon}\varepsilon ( {\bf k} + \overline{Y})\right\},\quad\quad\Lambda_\varepsilon=\Omega \setminus \widehat{\Omega}_\varepsilon,
\end{equation}
where $K_\varepsilon$ is the same as in Section \ref{Sec2}. Set
\begin{equation}\label{Omegahat*}
\widehat{\Omega}^\ast_\varepsilon=\widehat{\Omega}_\varepsilon\setminus S_\varepsilon,\quad\quad\Lambda^\ast_\varepsilon=\Omega^\ast_\varepsilon \setminus \widehat{\Omega}^\ast_\varepsilon.
\end{equation}
\noindent Moreover, throughout the paper we denote:
\begin{itemize}
\item $\widetilde{u}$: the zero extension to the whole $\Omega $ of a  function
$u$ defined on $\Omega ^{\ast}_{\varepsilon }$,
\item $\mathcal{M}_{E }\left( f\right) :=\dfrac{1}{\left\vert E
\right\vert }\displaystyle\int\nolimits_{E}f\de x$, the average on E of any function $f\in L^{1}\left(E\right)$.

\end{itemize}
Let us first recall the unfolding operator $\mathcal{T}_\varepsilon$ for the fixed domain $\Omega\times(0,T)$ introduced in \cite{Cio-Dam-Gri1} (see also \cite{Gaveau} where
the properties of $\mathcal{T}_\varepsilon$ are stated without proofs). Using the same notation as in \cite{Cio-Dam-Gri1}, let us give the following definition:
\begin{definition}\label{defunf}
Let $T>0$. For $p\in[1,+\infty)$ and $q\in [1+\infty]$, we define the operator $\mathcal{T}_{\varepsilon}:L^q(0,T;L^p(\Omega)) \rightarrow L^q(0,T;L^p(\Omega\times Y)) $ as follows:
\begin{equation*}
\mathcal{T}_{\varepsilon }\left( \varphi \right) \left( t,x,y\right)=\left\{
\begin{array}{ll}
\varphi \left( t, \varepsilon \left[ \dfrac{x}{\varepsilon }\right]%
_{Y}+\varepsilon y\right) & \, \, {\rm a.e.\ for }\left( t,x,y\right) \in (0,T) \times \widehat{\Omega}_\varepsilon \times Y\\
\\
0 & \, \, {\rm a.e.\ for }\left( t,x,y\right) \in (0,T) \times \Lambda_\varepsilon \times Y.
\end{array}
\right.
\end{equation*}
\end{definition}
Concerning perforated domains, we have the definition below (see \cite{Donato-Yang*}):
\begin{definition}\label{defunf*}
Let $T>0$. For $p\in[1,+\infty)$ and $q\in [1+\infty]$, we define the operator $\mathcal{T}^\ast_{\varepsilon}:L^q(0,T;L^p(\Omega^\ast_\varepsilon)) \rightarrow L^q(0,T;L^p(\Omega\times Y^\ast)) $ as follows:
\begin{equation*}
\mathcal{T}^\ast_{\varepsilon }\left( \varphi \right) \left( t,x,y\right)=\left\{
\begin{array}{ll}
\varphi \left( t, \varepsilon \left[ \dfrac{x}{\varepsilon }\right]%
_{Y}+\varepsilon y\right) & \, \, {\rm a.e.\ for }\left( t,x,y\right) \in (0,T) \times \widehat{\Omega}_\varepsilon \times Y^\ast\\
\\
0 & \, \, {\rm a.e.\ for }\left( t,x,y\right) \in (0,T) \times \Lambda_\varepsilon \times Y^\ast.
\end{array}
\right.
\end{equation*}
\end{definition}
Following the Remark 2.5 in \cite{Cio-Dam-Don-Gri-Zaki}, since the time variable acts as a simple parameter, the relationship between $\mathcal{T}_\varepsilon$ and $\mathcal{T}^\ast_\varepsilon$ is given, for any $\varphi$ defined on $(0,T)\times\Omega^\ast_\varepsilon$, by
\begin{equation}\label{link}
\mathcal{T}^\ast_\varepsilon(\varphi)=\mathcal{T}_\varepsilon(\widetilde{\varphi})_{|\Omega\times Y^\ast}.
\end{equation}
Actually, the previous equality still holds with every extension of $\varphi$ from $\Omega^\ast_\varepsilon$ into $\Omega$. In particular, for $\varphi$ defined on $\Omega$, we have
\begin{equation*}
\mathcal{T}^\ast_\varepsilon(\varphi_{|\Omega^\ast_\varepsilon})=\mathcal{T}_\varepsilon(\varphi)_{|\Omega\times Y^\ast}.
\end{equation*}
Hence, the operator $\mathcal{T}^\ast_\varepsilon$ inherits the properties of the operator  $\mathcal{T}_\varepsilon$ (see \cite{Donato-Yang*, Donato-Yang}) and for the reader's convenience they are recalled in the sequel.
\\
In particular, some immediate consequences of Definition \ref{defunf*} are:
\begin{enumerate}
\item[(i)] $\mathcal{T}^\ast_{\varepsilon }\left( \varphi \psi \right) =%
\mathcal{T}^\ast_{\varepsilon }\left( \varphi \right) \mathcal{T}^\ast_{\varepsilon }\left( \psi \right) $, for every $\varphi, \psi \in L^q(0,T;L^p(\Omega^\ast_\varepsilon))$;

\item[(ii)] $\mathcal{T}^\ast_{\varepsilon }\left( \varphi \psi \right) =%
\mathcal{T}^\ast_{\varepsilon }\left( \varphi \right)\psi$, for every $\varphi\in L^p(\Omega^\ast_\varepsilon)$ and $\psi\in  L^q(0,T)$;
\item[(iii)]$\nabla _{y}\left[ \mathcal{T}^\ast_{\varepsilon }\left(\varphi
\right) \right] =\varepsilon \mathcal{T}^\ast_{\varepsilon }\left(\nabla
\varphi \right)$ for every $\varphi \in L^q(0,T;W^{1,p}\left( \Omega^\ast_\varepsilon\right) )$;
\item[(iv)] $\mathcal{T}^\ast_\varepsilon\left(\varphi\left(t,\frac{x}{\varepsilon}\right)\right)=\varphi(t,y)$ a.e. in $(0,T)\times \Omega \times Y^\ast$ for any $Y-$periodic function
$\varphi \in L^q(0,T;L^p(Y^\ast))$;
\item[(v)] for all $\varphi\in L^q(0,T;L^p(\Omega^\ast_\varepsilon))$, we get
\begin{equation}\label{derunf}
\dfrac{\partial}{\partial t}(\mathcal{T}^\ast_{\varepsilon }(\varphi))(t,x,y)=\dfrac{\partial\varphi}{\partial t}\left(t, \varepsilon \left[ \dfrac{x}{\varepsilon }\right]%
_{Y}+\varepsilon y\right)=\mathcal{T}^\ast_{\varepsilon }\left(\dfrac{\partial\varphi}{\partial t}\right)(t,x,y)\,\text{ for a.e. }(t,x,y)\in (0,T)\times \Omega\times Y^\ast.
\end{equation}
\end{enumerate}
\begin{proposition}
\label{property}Let $T>0$. For $p\in \lbrack 1,+\infty \lbrack $ and $q\in [1,+\infty]$, let $\varphi^\varepsilon \in L^q(0,T; L^{1}\left( \Omega^\ast_\varepsilon\right))$ satisfying
\[\int_0^T\int_{\Lambda^\ast_\varepsilon}|\varphi^\varepsilon|\de x\de t \rightarrow 0.
\]
Then, one has
\begin{equation*}
\int_0^T\int\nolimits_{\Omega^\ast_\varepsilon}\varphi^\varepsilon \de x \de t-\frac{1}{\left\vert Y\right\vert }\int_0^T\int\nolimits_{\Omega \times Y^\ast}%
\mathcal{T}^\ast_{\varepsilon }\left( \varphi^\varepsilon \right) \de x \de y \de t\rightarrow 0.
\end{equation*}
As usual, this is denoted by
\begin{equation*}
\int_0^T\int\nolimits_{\Omega^\ast_\varepsilon}\varphi^\varepsilon \de x \de t \backsimeq \frac{1}{\left\vert Y\right\vert }\int_0^T\int\nolimits_{\Omega \times Y^\ast}%
\mathcal{T}^\ast_{\varepsilon }\left( \varphi^\varepsilon \right) \de x \de y \de t.
\end{equation*}
As a consequence, we have:
\begin{enumerate}
\item [(i)] for every $\varphi \in L^q(0,T;L^{p}\left( \Omega^\ast_\varepsilon\right) )$, one gets
\begin{equation*}
\left\Vert \mathcal{T}^\ast_{\varepsilon }\left( \varphi \right)
\right\Vert_{L^q(0,T;\, L^{p}\left( \Omega \times Y^\ast\right) )}\leq \left\vert
Y\right\vert^{1/p}\left\Vert \varphi \right\Vert _{L^q(0,T;\, L^{p}\left( \Omega^\ast_\varepsilon\right)) },
\end{equation*}
which means that the operator $\mathcal{T}^\ast_{\varepsilon }$ is continuous from $L^q(0,T;L^p(\Omega^\ast_\varepsilon))$ to $L^q(0,T;L^p(\Omega\times Y^\ast)) $;
\item[(ii)] for every $\varphi \in L^q(0,T;W^{1,p}\left( \Omega^\ast_\varepsilon\right) )
$, one has $\mathcal{T}^\ast_{\varepsilon }\left(\varphi\right)
\in L^q(0,T;L^{2}\left( \Omega,W^{1,p}\left( Y^\ast\right) )\right);
$
\item[(iii)] for $p, q\in (1,+\infty]$, let $\varphi^\varepsilon\in L^q(0,T; L^p(\Omega^\ast_\varepsilon))$ and $\psi\in L^{q'}(0,T; L^{p'}(\Omega^\ast_\varepsilon))$, with $\dfrac{1}{p}+\dfrac{1}{p'}=1$, $\dfrac{1}{q}+\dfrac{1}{q'}=1$ such that
\begin{equation}
\|\varphi_\varepsilon\|_{L^q(0,T; L^p(\Omega^\ast_\varepsilon))}\leq C\quad\text{ and }\quad \|\psi\|_{L^{q'}(0,T; L^{p'}(\Omega^\ast_\varepsilon))}\leq C
\end{equation}
with $C$ a positive constant independent of $\varepsilon$. Then,
\begin{equation*}
\int_0^T\int\nolimits_{\Omega^\ast_\varepsilon}\varphi^\varepsilon\psi \de x \de t \backsimeq \frac{1}{\left\vert Y\right\vert }\int_0^T\int\nolimits_{\Omega \times Y^\ast}%
\mathcal{T}^\ast_{\varepsilon }\left( \varphi_\varepsilon \right)\mathcal{T}^\ast_{\varepsilon }\left( \psi\right) \de x \de y \de t;
\end{equation*}
\item[(iv)] for $p, q\in (1,+\infty]$, let $\varphi^\varepsilon\in L^q(0,T; L^p(\Omega^\ast_\varepsilon))$ and $\psi^\varepsilon\in L^{q'}(0,T; L^{p_0}(\Omega^\ast_\varepsilon))$, with $\dfrac{1}{p}+\dfrac{1}{p_0}<1$, $\dfrac{1}{q}+\dfrac{1}{q'}=1$ such that
\begin{equation}
\|\varphi^\varepsilon\|_{L^q(0,T; L^p(\Omega^\ast_\varepsilon))}\leq C\quad\text{ and }\quad \|\psi^\varepsilon\|_{L^{q'}(0,T; L^{p_0}(\Omega^\ast_\varepsilon))}\leq C
\end{equation}
with $C$ a positive constant independent of $\varepsilon$. Then
\begin{equation*}
\int_0^T\int\nolimits_{\Omega^\ast_\varepsilon}\varphi^\varepsilon\psi^\varepsilon \de x \de t \backsimeq \frac{1}{\left\vert Y\right\vert }\int_0^T\int\nolimits_{\Omega \times Y^\ast}%
\mathcal{T}^\ast_{\varepsilon }\left( \varphi^\varepsilon \right)\mathcal{T}^\ast_{\varepsilon }\left( \psi^\varepsilon\right) \de x \de y \de t.
\end{equation*}
\end{enumerate}
\end{proposition}
Moreover, we have the following convergence properties.
\begin{proposition}
\label{convergence}
\begin{enumerate}
\item[(i)] Let $\varphi \in L^q(0,T;L^{p}\left( \Omega \right) )$. Then,
\begin{equation*}
\mathcal{T}^\ast_{\varepsilon }\left( \varphi \right) \longrightarrow \varphi
\text{ \ strongly in }L^q(0,T;L^{p}\left( \Omega \times Y^\ast\right)) \text{.}
\end{equation*}
\item[(ii)] Let $\varphi ^{\varepsilon }\in
L^q(0,T;L^{p}\left( \Omega \right)) $ such that $\varphi
^{\varepsilon}\longrightarrow \varphi $ strongly in $L^q(0,T;L^{p}\left( \Omega
\right)) $. Then,
\begin{equation*}
\mathcal{T}^\ast_{\varepsilon }\left( \varphi ^{\varepsilon
}\right)\longrightarrow \varphi \text{ \ strongly in }L^q(0,T;L^{p}\left( \Omega
\times Y^\ast\right)) \text{.}
\end{equation*}
\item[(iii)] Let $\varphi ^{\varepsilon }\in L^q(0,T;L^{p}\left(
\Omega^\ast_\varepsilon\right)) $ satisfy $\left\Vert \varphi
^{\varepsilon}\right\Vert _{L^q(0,T;L^{p}\left(
\Omega^\ast_\varepsilon\right))}\leq C$ and
\[
\mathcal{T}^\ast_{\varepsilon }\left( \varphi ^{\varepsilon
}\right) \rightharpoonup \widehat{\varphi }\,\text{ weakly in } L^q(0,T;L^{p}\left( \Omega
\times Y^\ast\right)).
\]
Then,
\begin{equation*}
\widetilde{\varphi }^{\varepsilon }\rightharpoonup \dfrac{\vert Y^\ast\vert}{\vert Y\vert }\mathcal{M}%
_{Y^\ast}\left( \widehat{\varphi }\right) \text{ \ weakly in }%
L^q(0,T;L^{p}\left(\Omega \right) ).
\end{equation*}
\end{enumerate}
\end{proposition}
Let us finally recall a known result about the convergences of the previously introduced unfolding operators $\mathcal{T}_\varepsilon$ and $\mathcal{T}^\ast_\varepsilon$, applied to bounded sequences in $H^1(\Omega)$ and $H^1(\Omega^\ast_\varepsilon)$, respectively.
\begin{theorem}\label{u1}
Let $v_\varepsilon$ be a sequence in $L^2(0,T;H^1(\Omega))$ such that
\[
\|v_\varepsilon\|_{L^2(0,T;H^1(\Omega))}\leq C,
 \]
with $C$ a positive constant independent of  $\varepsilon$. Then, there exist $v\in L^2(0,T;H^1(\Omega))$ and $\widehat{v}\in L^2 \left (0,T;L^{2}( \Omega,H_{per}^1(Y)/\Bbb R)\right)$ such that, up to a subsequence,
\begin{equation}
\left\{
\begin{array}{lll}
\mathcal{T}_{\varepsilon }\left( v_{\varepsilon }\right)
\rightharpoonup v & \text{weakly in} & L^2\left (0,T;L^{2}( \Omega,H^{1}\left(
Y\right) )\right), \\[2 mm]
\mathcal{T}_{\varepsilon }\left( \nabla v_{\varepsilon }\right)
\rightharpoonup \nabla v+\nabla _{y}\widehat{v}& \text{weakly in} &
L^2(0,T;L^{2}\left( \Omega \times Y\right)).
\end{array}%
\right.  \label{8}
\end{equation}
\end{theorem}
\begin{theorem}\label{u1*}
Let $v_\varepsilon$ be a sequence in $L^2(0,T;H^1(\Omega^\ast_\varepsilon))$ such that
\[
\|v_\varepsilon\|_{L^2(0,T;H^1(\Omega^\ast_\varepsilon))}\leq C,
 \]
with $C$ a positive constant independent of  $\varepsilon$. Then, there exist $v\in L^2(0,T;H^1(\Omega))$ and $\widehat{v}\in L^2 \left (0,T;L^{2}( \Omega,H_{per}^1(Y^\ast)/\Bbb R)\right)$ such that, up to a subsequence,
\begin{equation}
\left\{
\begin{array}{lll}
\mathcal{T}^\ast_{\varepsilon }\left( v_{\varepsilon }\right)
\rightharpoonup v & \text{weakly in} & L^2\left (0,T;L^{2}( \Omega,H^{1}\left(
Y^\ast\right) )\right), \\[2 mm]
\mathcal{T}^\ast_{\varepsilon }\left( \nabla v_{\varepsilon }\right)
\rightharpoonup \nabla v+\nabla _{y}\widehat{v}& \text{weakly in} &
L^2(0,T;L^{2}\left( \Omega \times Y^\ast\right)).
\end{array}%
\right.  \label{8}
\end{equation}
\end{theorem}
Using the same notation as in \cite{Cio-Dam-Don-Gri-Zaki} (see, also, \cite{Cabarrubias-Donato} and \cite{Donato-Yang}), let us give the following definition:
\begin{definition}\label{defunfbound}
Let $T>0$. For any function $\varphi$ which is Lebesgue measurable on $\Gamma_{\varepsilon}$, we define the boundary unfolding operator $\mathcal{T}^{b}_{\varepsilon}$ as follows:
\[
\mathcal{T}^{b}_{\varepsilon }\left( \varphi \right) \left(t, x,y\right)=\varphi \left(t, \varepsilon \left[ \dfrac{x}{\varepsilon }\right]%
_{Y}+\varepsilon y\right) \quad for \,\, a.e.\ \left(t,x,y\right) \in (0,T)\times \Omega \times \Gamma.
\]
\end{definition}
\begin{proposition}
\label{propertyb}Let $p,q\in \lbrack 1,+\infty \lbrack $ and $T>0$.  The operator $%
\mathcal{T}^b_{\varepsilon }$ is linear
and continuous from $L^q(0,T;L^p(\Gamma^\varepsilon))$ to $%
L^q(0,T;L^p(\Omega \times \Gamma)) $. Moreover,

\begin{enumerate}
\item[(i)] For every $\varphi\in L^q(0,T; L^{1}\left( \Gamma^\varepsilon)\right)$, one gets
\begin{equation}\label{intbound}
\dfrac{1}{\varepsilon |Y|}\int_{\Omega\times\Gamma}T^b_\varepsilon(\varphi)(t,x,y) \de x \de \sigma_y=\int_{\Gamma_\varepsilon}\varphi(t,x) \de \sigma_x,
\end{equation}
for a.e. $t\in (0,T)$.
\item[(ii)] For every $\varphi \in L^q(0,T;L^{p}\left(  \Gamma^\varepsilon)\right) $, one gets
\begin{equation}\label{normbound}
\|T^b_\varepsilon(\varphi)\|_{L^q(0,T;L^{p}\left( \Omega\times\Gamma\right) )}\leq |Y|^{1/p}\varepsilon^{1/p}\|\varphi\|_{L^q(0,T;L^{p}\left( \Gamma^\varepsilon\right) )}.
\end{equation}
\end{enumerate}
\end{proposition}
\begin{remark}\label{remgeo}
We shall be interested in working with these unfolding operators only for our particular form of the domain $\Omega$ (see Section \ref{Sec2}). Hence, for such a geometry, it holds $\widehat{\Lambda}_\varepsilon=\widehat{\Lambda}^\ast_\varepsilon=\emptyset$.
\end{remark}
\subsection{Weak and strong convergence results}
Under the assumptions we imposed on the geometry and on the data, we can use extensions for time-dependent functions to the whole of the domain $\Omega$. Following \cite{Bohm} and \cite{Hopker} (see, also, \cite{Acerbi}, \cite{Cio-Pau}, \cite{Gahn}, and \cite{Meirmanov}), in our geometry, for $i\in \{1,2,3\}$, there exists a linear and bounded extension operator ${\cal L}_i^\varepsilon:
L^2(0,T; H^1_{\partial \Omega}(\Omega^\ast_\varepsilon)) \rightarrow L^2(0,T; H_0^1(\Omega))$. We denote
\begin{equation}\label{extension}
{\cal L}_i^\varepsilon (c_i^\varepsilon) =\overline{c}_i^\varepsilon.
\end{equation}
We remark that the above linear and bounded extension operator to the whole of $\Omega$ preserves the non-negativity,
the essential boundedness and the {\it a priori} estimates \eqref{estimconcLinf}-\eqref{estimconcb} obtained for the solution $(c_{1}^{\varepsilon}, c_{2}^{\varepsilon}, c_{3}^{\varepsilon})$. Moreover, as in \cite{Gra-Pet} and \cite{Graf}, it follows that $$\overline{c}_i^\varepsilon  \in L^2((0,T), H_0^1(\Omega))\cap H^1((0,T),(H^1_0(\Omega))^{'})\cap L^\infty((0,T)\times \Omega),$$ with bounds independent of $\varepsilon$, and there exists $c'_i\in L^2(0,T; H_0^1 (\Omega))$ such that, for $i\in \{1,2,3\}$,
\begin{equation}\label{convext}
\overline{c}_i^\varepsilon  \rightarrow c'_i\,\,\,\,\text{strongly in }L^2((0, T) \times \Omega).
\end{equation}

Let us fix $i\in\{1,2,3\}$. For the function $c^\varepsilon_i\in \mathcal{W}\left(0,T;H^{1}_{\partial \Omega}(\Omega^{\ast}_{\varepsilon}),(H^{1}_{\partial \Omega}(\Omega^{\ast}_{\varepsilon}))'\right)$, we consider the time derivative $\partial_t \tilde{c}_i^\varepsilon \in L^2(0,T;(H_0^1(\Omega))')$ of the extension by zero of $c_i^\varepsilon$. It is obvious that the generalized time derivative of $\tilde{c}_i^\varepsilon$ exists and it holds
\begin{equation}\label{dergen}
\langle\partial_{t}\tilde{c}_{i}^{\varepsilon}(t),v\rangle_{\Omega}=\langle\partial_{t}c_{i}^{\varepsilon}(t),v|_{\Omega^{\ast}_{\varepsilon}}\rangle_{\Omega^{\ast}_{\varepsilon}}\,\,\forall v\in H_0^1(\Omega)\text{ and a.e. }t\in(0,T),
\end{equation}
which implies
\begin{equation}\label{dergenbis}
\|\partial_t \tilde{c}_i^\varepsilon\|_{L^2(0,T; (H_0^1(\Omega))')}\leq \|\partial_t c_i^\varepsilon\|_{L^2(0,T; (H_0^1(\Omega^{\ast}_\varepsilon))')} .
\end{equation}
In the next lemma, we collect the main compactness results we have for the solution of our microscopic problem \eqref{eqmicro} obtained by using the properties of the time-dependent unfolding operator $\mathcal{T}^\ast_{\varepsilon}$
for perforated domains recalled in Section \ref{sub2.3} and the {\it a priori} estimates proved in Lemma \ref{lemest}.
\begin{lemma}\label{lemma-conv}
Let $(c_1^\varepsilon, c_2^\varepsilon, c_3^\varepsilon)$ be the unique solution of problem \eqref{weakconcentration}-\eqref{weakinitial}. Then, up to a subsequence, there exist $c$ and $c_3\in L^2( 0,T; H_0^{1}(\Omega))$,
$\widehat{c}_i\in L^2((0,T) \times \Omega; H^1_{\textrm{per}}(Y^\ast)/\mathbb R)$, with $i\in \{1,2,3\}$,
such that, for $\varepsilon \rightarrow 0$, we have
\begin{equation}\label{conv-genc12}
\left\{
\begin{array}{lll}
i)&\mathcal{T}^\ast_{\varepsilon } (c^{\varepsilon }_i) \rightharpoonup c& \textrm{weakly in } L^2((0,T) \times \Omega, H^1(Y^\ast)),\\
\\
ii)&\mathcal{T}^\ast_{\varepsilon } (\nabla c^{\varepsilon }_i) \rightharpoonup \nabla c+ \nabla_y \widehat{c}_i
& \textrm{weakly  in } L^2((0,T) \times \Omega \times Y^\ast),\\
\\
iii)&\mathcal{T}^\ast_{\varepsilon } (c_i^\varepsilon) \rightarrow c & \textrm{strongly in } L^2((0,T)\times \Omega\times Y^\ast),\\
\\
iv)&{\cal T}_\varepsilon^b (c^\varepsilon_i) \rightarrow c& \textrm{strongly in } L^2((0,T)\times \Omega \times \Gamma),\\
\\
v)&\partial_t \tilde{c}_i^\varepsilon \rightharpoonup \vert Y^* \vert \partial_t c &  \textrm{weakly in } L^2(0,T; (H_0^1(\Omega))'),
\end{array}
\right.
\end{equation}
for $i\in\{1,2\}$ and
\begin{equation}\label{conv-genc3}
\left\{
\begin{array}{lll}
i)&\mathcal{T}^\ast_{\varepsilon } (c^{\varepsilon }_3) \rightharpoonup c_3& \textrm{weakly in } L^2((0,T) \times \Omega, H^1(Y^\ast)),\\
\\
ii)&\mathcal{T}^\ast_{\varepsilon } (\nabla c^{\varepsilon }_3) \rightharpoonup \nabla c_3+ \nabla_y \widehat{c}_3
& \textrm{weakly  in } L^2((0,T) \times \Omega \times Y^\ast),\\
\\
iii)&\mathcal{T}^\ast_{\varepsilon } (c_3^\varepsilon) \rightarrow c_3 & \textrm{strongly in } L^2((0,T)\times \Omega\times Y^\ast),\\
\\
iv)&{\cal T}_\varepsilon^b (c^\varepsilon_3) \rightarrow c_3& \textrm{strongly in } L^2((0,T)\times \Omega \times \Gamma),\\
\\
v)&\partial_t \tilde{c}_3^\varepsilon \rightharpoonup \vert Y^* \vert \partial_t c_3 &  \textrm{weakly in } L^2(0,T; (H_0^1(\Omega))').
\end{array}
\right.
\end{equation}

\end{lemma}
\begin {proof}
Let us fix $i\in\{1,2,3\}$. By \eqref{estimconcH1} of Lemma \ref{lemest} and Theorem \ref{u1*}, it follows that  there exist $c_i\in L^2( 0,T; H_0^{1}(\Omega))$ and
$\widehat{c}_i\in L^2((0,T) \times \Omega; H^1_{\textrm{per}}(Y^\ast)/\mathbb R)$ such that, up to a subsequence still denoted by $\varepsilon$, the following convergences hold
\begin{equation}\label{conv1}
\left\{
\begin{array}{lll}
i)&\mathcal{T}^\ast_{\varepsilon } (c^{\varepsilon }_i) \rightharpoonup c_i& \textrm{weakly in } L^2((0,T) \times \Omega, H^1(Y^\ast)),\\
\\
ii)&\mathcal{T}^\ast_{\varepsilon } (\nabla c^{\varepsilon }_i) \rightharpoonup \nabla c_i+ \nabla_y \widehat{c}_i
& \textrm{weakly  in } L^2((0,T) \times \Omega \times Y^\ast).\\
\end{array}
\right.
\end{equation}
On the other hand, by Theorem \ref{u1} and \eqref{convext}, we get
\begin{equation*}
T_\varepsilon(\overline{c}_i^\varepsilon)\to c'_i \quad{ \rm in }\quad L^2((0,T)\times\Omega\times Y),
\end{equation*}
which implies
\begin{equation*}
T_\varepsilon(\overline{c}_i^\varepsilon)_{|\Omega\times Y^\ast} \to (c'_i)_{|\Omega\times Y^\ast}=c'_i\quad{ \rm in }\quad L^2((0,T)\times\Omega\times Y^\ast),
\end{equation*}
since $c'_i$ doesn't depends on $y$. On the other hand, by \eqref{link} it holds
$$
T_\varepsilon(\overline{c}_i^\varepsilon)_{|\Omega\times Y^\ast}=T^{\ast}_\varepsilon(c_i^\varepsilon).
$$
Hence, we can deduce that
\begin{equation*}
T^{\ast}_\varepsilon(c_i^\varepsilon) \to c'_i \quad{\rm in }\quad L^2((0,T)\times\Omega\times Y^\ast).
\end{equation*}
Due to \eqref{conv1}i), by uniqueness $c'_i=c_i$ and we get indeed
\begin{equation} \label{strong-domain}
\mathcal{T}^\ast_{\varepsilon } (c_i^\varepsilon) \rightarrow c_i \quad \textrm{strongly in  }
L^2((0,T)\times \Omega\times Y^\ast).
\end{equation}
Moreover, by Lemma \ref{lemest}, acting as in \cite[Proposition 12]{Gahn}, we get
\begin{equation} \label{strong-boundary}
\mathcal{T}^{b}_{\varepsilon }  (c_i^\varepsilon) \rightarrow c_i \quad \textrm{strongly in } L^2((0,T)\times \Omega \times \Gamma).
\end{equation}
We remark that, for $i\in \{1,2,3\}$, $c_i\in \mathcal{W}\left(0,T;H_0^1(\Omega), (H_0^1(\Omega))'\right)$.
Now it remains to prove that $c_1=c_2=c$ in $(0,T)\times \Omega$. To this, let us observe that, by \eqref{estimdifconc} and \eqref{normbound}, we get
\begin{equation}\label{c1=c2}
\begin{array}{l}
\|{\cal T}_\varepsilon^b (c^\varepsilon_2)-c_1\|_{L^2((0,T)\times\Omega\times\Gamma)}\leq \|{\cal T}_\varepsilon^b (c^\varepsilon_2)-{\cal T}_\varepsilon^b (c^\varepsilon_1)\|_{L^2((0,T)\times\Omega\times\Gamma)}+\|{\cal T}_\varepsilon^b (c^\varepsilon_1)-c_1\|_{L^2((0,T)\times\Omega\times\Gamma)}\\
\\
\leq \sqrt\varepsilon\|c_2^\varepsilon-c_1^\varepsilon\|_{L^2((0,T)\times\Gamma_\varepsilon)}+\|{\cal T}_\varepsilon^b (c^\varepsilon_1)-c_1\|_{L^2((0,T)\times\Omega\times\Gamma)}\\
\\
\leq C\varepsilon+\|{\cal T}_\varepsilon^b (c^\varepsilon_1)-c_1\|_{L^2((0,T)\times\Omega\times\Gamma)}.
\end{array}
\end{equation}
Hence, when $\varepsilon$ tends to zero, by \eqref{strong-boundary} it holds
\begin{equation*}
\mathcal{T}^{b}_{\varepsilon }  (c_2^\varepsilon) \rightarrow c_1\quad \textrm{strongly in } L^2((0,T)\times \Omega \times \Gamma),
\end{equation*}
which means that $c_1=c_2$ in $(0,T)\times \Omega \times \Gamma$. Since $c_1$ and $c_2$ are independent of $y$, we obtain $c_1=c_2$ in the whole $(0,T)\times \Omega$.\\
By  \eqref{estimconcH1'} and \eqref{dergenbis}, there exists $W_i \in L^2(0,T; (H_0^1(\Omega))')$ such that, up to a subsequence, we obtain
\begin{equation}
\partial_t \tilde{c}_i^\varepsilon \rightharpoonup W_i \quad
\textrm{weakly in } L^2(0,T; (H_0^1(\Omega))').
\end{equation}
An easy integration by parts, Proposition \ref{convergence}(iii) and \eqref{conv1} show that $W_i=\dfrac{\vert Y^\ast \vert}{\vert Y\vert} \partial_t c_i $. Hence,
\begin{equation}\label{derconv}
\partial_t \tilde{c}_i^\varepsilon \rightharpoonup \dfrac{\vert Y^\ast \vert}{\vert Y\vert}  \partial_t c_i \quad
\textrm{weakly in } L^2(0,T; (H_0^1(\Omega))').
\end{equation}
Finally, \eqref{conv1}, \eqref{strong-domain}, \eqref{strong-boundary}, and \eqref{derconv} imply \eqref{conv-genc12} for $i\in\{1,2\}$ and \eqref{conv-genc3} for $i=3$.
\end{proof}
\begin{remark}\label{rem3}
In order to use classical compactness results and to obtain convergences of the microscopic solution, we are forced to extend, at first, the functions $(c_{1}^{\varepsilon}, c_{2}^{\varepsilon}, c_{3}^{\varepsilon})$ to the whole domain $\Omega$ by means of a suitable uniform extension operator before unfolding. Indeed, by using the a priori uniform estimate \eqref{estimconcH1} of Lemma \ref{lemest} and Theorem \ref{u1*} we could get only weak convergences of $c_i^\varepsilon$  and $T^\ast_\varepsilon(c_i^\varepsilon)$, $i\in\{1,2,3\}$. Moreover, due to the less regularity of the time derivative, these weak convergences can't be improved unlike in \cite{Donato-Yang*}. When handling nonlinear terms, as in our paper, weak convergence isn't enough, but we need also strong convergence with respect to suitable topologies. More generally, we can deduce that when dealing with the homogenization by unfolding in a perforated domain, if there exists a classical uniform extension operator, it is like we could act directly with $T^\ast_\varepsilon$.
On the other hand, if we cannot construct such a uniform extension operator, due to some particular reasons (for example, the lack of regularity of the boundary of the holes), we can homogenize as well (this is the main advantage of the unfolding), but, in the presence of nonlinear terms, we are forced to prove a convergence like \eqref{strong-domain}.
\end{remark}

\subsection{The macroscopic model}
The main convergence result of this paper is stated in the next theorem, where we take into account that $\vert Y\vert=1$.
\begin{theorem}\label{teounf} Let $(c_{1}^{\varepsilon}, c_{2}^{\varepsilon}, c_{3}^{\varepsilon})$ be the solution of system  \eqref{weakconcentration}-\eqref{weakinitial}. Then, under the assumptions $(\mathbf{H}_1)\div (\mathbf{H})_7$, there exist $c$ and $c_{3}\in
L^{2}(0,T;H_0^{1}(\Omega))$ and $\widehat{c}_{i}\in L^{2}((0,T)\times\Omega;H_{per}^{1}(Y^\ast)/\Bbb R)$, $i\in \{1,2,3\}$, such that
\begin{equation} \label{unfold-convconc}%
\left\{
\begin{array}{lll}
i)&\mathcal{T}^{\ast}_{\varepsilon}(c_{i}^{\varepsilon})\rightharpoonup c &\text{\rm weakly  in }L^{2}((0,T)\times\Omega;H^{1}(Y^\ast)),\\[2mm]
ii)&\mathcal{T}^{\ast}_{\varepsilon}(\nabla c_{i}^{\varepsilon})\rightharpoonup\nabla c+\nabla_{y}\widehat{c}_{i}&\text{\rm weakly in }L^{2}((0,T)\times
\Omega\times Y^\ast),\\[2mm]
iii)&\mathcal{T}^\ast_{\varepsilon } (c_i^\varepsilon) \rightarrow c & \text{\rm strongly in } L^2((0,T)\times \Omega\times Y^\ast),
\end{array}
\right.
\end{equation}
for $i\in\{1,2\}$ and
\begin{equation} \label{unfold-convconc3}%
\left\{
\begin{array}{lll}
i)&\mathcal{T}^{\ast}_{\varepsilon}(c_{3}^{\varepsilon})\rightharpoonup c_{3}&\text{\rm weakly  in }L^{2}((0,T)\times\Omega;H^{1}(Y^\ast)),\\[2mm]
ii)&\mathcal{T}^{\ast}_{\varepsilon}(\nabla c_{3}^{\varepsilon})\rightharpoonup\nabla c_{3}+\nabla_{y}\widehat{c}_{3}&\text{\rm weakly in }L^{2}((0,T)\times
\Omega\times Y^\ast),\\[2mm]
iii)&\mathcal{T}^\ast_{\varepsilon } (c_3^\varepsilon) \rightarrow c_3 & \text{\rm strongly in } L^2((0,T)\times \Omega\times Y^\ast),
\end{array}
\right.
\end{equation}
where $(c, \widehat{c_1}, \widehat{c}_{2}, c_3, \widehat{c}_{3})$ is the unique solution of the following problem

\begin{equation}\label{probunf}
\left\{
\begin{array}{l}
\text{\rm Find } (c, c_3)\in (L^{2}(0,T;H_0^{1}(\Omega)))^2 \ {\text{\rm  and   }} \widehat{c}_{i}\in L^{2}((0,T)\times\Omega;H_{per}^{1}(Y^\ast)/\Bbb R),\,i\in \{1,2,3\}, \,\text{\rm  such that  } \\
\\
2\displaystyle \vert Y^\ast \vert \, \langle\partial_t c, \varphi\rangle_{\Omega}+
\displaystyle \int_{\Omega\times Y^{\ast}}D_1(y)(\nabla c+\nabla_{y}\widehat{c_{1}})(\nabla\varphi+\nabla_{y}\Psi_1) \de x \de y\\
\\
+\displaystyle \int_{\Omega\times Y^{\ast}}D_2(y)(\nabla c+\nabla_{y}\widehat{c_{2}})(\nabla\varphi+\nabla_{y}\Psi_2)\,\mathrm{d}x\mathrm{d}y +\displaystyle \int_{\Omega\times\Gamma} (\widehat{c_{1}}-\widehat{c_{2}}) H(c_3) (\Psi_1-\Psi_2) \de x \de \sigma_{y}\\
\\
\displaystyle = \int_{\Omega\times Y^{\ast}}F_{1}(y,c,c, c_{3})\varphi \de x \de y+\int_{\Omega\times Y^{\ast}}F_{2}(y,c,c, c_{3})\varphi \de x \de y,\\
\\
\displaystyle \vert Y^\ast \vert \, \langle\partial_t c_3, \varphi\rangle_{\Omega}+
\displaystyle \int_{\Omega\times Y^{\ast}}D_3(y)(\nabla c_{3}+\nabla_{y}\widehat{c_{3}})(\nabla\varphi+\nabla_{y}\Psi_3) \de x \de y\\
\\
\displaystyle=\int_{\Omega\times\Gamma}G_{3}(y,c,c,c_{3})\varphi \de x \de\sigma_{y}+
\int_{\Omega\times Y^{\ast}}F_{3}(y,c,c, c_{3})\varphi \de x \de y,\\
\\
\text{\rm in  }\mathcal{D}'(0,T)\,\text{\rm and for all } \varphi\in H_0^{1}(\Omega),\ \Psi_i\in L^2(\Omega
;H_{per}^{1}(Y^\ast)),\\
\\
c(x,0)=(c_1^0+c_2^0)/2 \,\,\,\,\,\,\text {\rm in } \Omega,\\
\\
c_3(x,0)=c_3^0\,\,\,\,\,\,\,\,\,\,\,\,\,\,\,\,\,\,\,\,\,\,\,\,\,\,\text{ \rm in } \Omega.
\end{array}
\right.
\end{equation}
\end{theorem}
\begin{proof} By Lemma \ref{lemma-conv}, we get convergences \eqref{unfold-convconc} and \eqref{unfold-convconc3}, up to a subsequence, still denoted by $\varepsilon$. It remains to prove that  $(c, \widehat{c}_{1}, \widehat{c}_{2}, c_3, \widehat{c}_{3})$ is solution of the limit problem \eqref{probunf}. To this aim, let
\begin{equation}\label{test}
v_i^\varepsilon(x)=\varphi(x) +\varepsilon \, \omega_i(x) \, \psi_i^\varepsilon(x),
\end{equation}
where $\varphi, \omega_i \in \mathcal{D}(\Omega)$, $\psi_i^\varepsilon(x)=\psi_i\left(\dfrac{x}{\varepsilon}\right)$, and $\psi_i\in H_{\textrm{per}}^{1}(Y^\ast)$.\\

Since $\mathcal{T}^{\ast}_{\varepsilon}\left(\nabla v_i^\varepsilon \right)=\mathcal{T}^{\ast}_{\varepsilon}\left(\varphi\right) +\varepsilon\mathcal{T}^{\ast}_{\varepsilon}\left(\nabla\omega\right)
\mathcal{T}^{\ast}_{\varepsilon}\left(\psi_i^\varepsilon\right)+
\mathcal{T}^{\ast}_{\varepsilon}\left(\omega\right)\mathcal{T}^{\ast}_{\varepsilon}\left(\nabla_y\psi_i^\varepsilon\right)$, we have
\begin{equation}\label{convtest}
\begin{array}{ll}
\mathcal{T}^{\ast}_{\varepsilon}\left(v_i^\varepsilon \right)\rightarrow \varphi&\text{strongly in }L^2(\Omega\times Y^\ast),\\
\\
\mathcal{T}^{\ast}_{\varepsilon}\left(\nabla v_i^\varepsilon \right)\rightarrow \nabla\varphi+\omega_i\nabla_y \psi_i&\text{strongly in }L^2(\Omega\times Y^\ast).
\end{array}
\end{equation}

In order to get the first equation in \eqref{probunf}, let us take $v=v_1^\varepsilon$ and $v=v_2^\varepsilon$ as test functions in the first equation in \eqref{weakconcentration}, written for $i=1$ and $i=2$, respectively. Multiplying these equations by $w\in \mathcal{D}(0,T)$, integrating by parts and summing them up, we get

\begin{equation*}
\begin{array}{l}
-\displaystyle\int_0^T\int_{\Omega^{\ast}_{\varepsilon}}c_1^\varepsilon  \, v_1^\varepsilon w\, ' \mathrm{d}x\,\mathrm{d}t-\displaystyle\int_0^T\int_{\Omega^{\ast}_{\varepsilon}}c_2^\varepsilon \, v_2^\varepsilon w\,' \de x \de t\\
\\
+ \displaystyle \int_0^T\int_{\Omega^{\ast}_{\varepsilon}}D_1^\varepsilon \nabla c_{1}^{\varepsilon}\cdot\nabla v_1^{\varepsilon}\, w \de x \de t + \displaystyle \int_0^T\int_{\Omega^{\ast}_{\varepsilon}}D_2^\varepsilon \nabla c_{2}^{\varepsilon}\cdot\nabla v_2^{\varepsilon}\, w \de x \de t\\
\\
=\displaystyle \frac{1}{\varepsilon}\int_0^T\int_{\Gamma^{\varepsilon}}(G_{1}(c_{1}%
^{\varepsilon},c_{2}^{\varepsilon}, c_3^{\varepsilon})v_1^{\varepsilon} +G_{2}(c_{1}%
^{\varepsilon},c_{2}^{\varepsilon}, c_3^{\varepsilon})v_2^{\varepsilon}) \, w \de \sigma
_{x} \de t\\
\\
+\displaystyle \int_0^T\int_{\Omega
^{\varepsilon}}F_{1}^{\varepsilon}(x,c_{1}^{\varepsilon},c_{2}^{\varepsilon
}, c_3^{\varepsilon})v_1^{\varepsilon} \, w \de x \de t +\displaystyle \int_0^T\int_{\Omega
^{\varepsilon}}F_{2}^{\varepsilon}(x,c_{1}^{\varepsilon},c_{2}^{\varepsilon
}, c_3^{\varepsilon})v_2^{\varepsilon} \, w \de x \de t.
\end{array}
\end{equation*}
By \eqref{formG} we obtain
\begin{equation}\label{unf3}
\begin{array}{l}
-\displaystyle\int_0^T\int_{\Omega^{\ast}_{\varepsilon}}c_1^\varepsilon  \, v_1^\varepsilon w\, ' \mathrm{d}x\,\mathrm{d}t-\displaystyle\int_0^T\int_{\Omega^{\ast}_{\varepsilon}}c_2^\varepsilon \, v_2^\varepsilon w\,' \de x \de t
+ \displaystyle \int_0^T\int_{\Omega^{\ast}_{\varepsilon}}D_1^\varepsilon \nabla c_{1}^{\varepsilon}\cdot\nabla v_1^{\varepsilon}\, w \de x \de t \\
\\
+ \displaystyle \int_0^T\int_{\Omega^{\ast}_{\varepsilon}}D_2^\varepsilon \nabla c_{2}^{\varepsilon}\cdot\nabla v_2^{\varepsilon}\, w \de x \de t
+\displaystyle\int_0^T\int_{\Gamma^{\varepsilon}}(c_{1}
^{\varepsilon}-c_{2}^{\varepsilon})H\left (c_3^{\varepsilon}\right )(\omega_1 \, \psi_1^\varepsilon- \omega_2 \, \psi_2^\varepsilon) \, w \de \sigma
_{x} \de t\\
\\
=\displaystyle \int_0^T\int_{\Omega^\ast
_{\varepsilon}}F_{1}^{\varepsilon}(x,c_{1}^{\varepsilon},c_{2}^{\varepsilon
}, c_3^{\varepsilon})v_1^{\varepsilon} \, w \de x \de t +\displaystyle \int_0^T\int_{\Omega^\ast
_{\varepsilon}}F_{2}^{\varepsilon}(x,c_{1}^{\varepsilon},c_{2}^{\varepsilon
}, c_3^{\varepsilon})v_2^{\varepsilon} \, w \de x \de t.
\end{array}
\end{equation}
By Proposition \ref{property} and Proposition \ref{propertyb}, we unfold \eqref{unf3} by means of the operators $\mathcal{T}^{\ast}_{\varepsilon}$ and $\mathcal{T}^{b}_{\varepsilon}$ and by assumption $(\mathbf{H}_1)_{1}$ we have
\begin{equation}\label{unf3bis}
\begin{array}{l}
-\displaystyle\int_0^T\int_{\Omega\times Y^\ast}\mathcal{T}^{\ast}_{\varepsilon}(c_1^\varepsilon)  \, \mathcal{T}^{\ast}_{\varepsilon}(v_1^\varepsilon) w\, ' \mathrm{d}x\,\de y\,\mathrm{d}t-\displaystyle\int_0^T\int_{\Omega\times Y^\ast}\mathcal{T}^{\ast}_{\varepsilon}(c_2^\varepsilon) \, \mathcal{T}^{\ast}_{\varepsilon}(v_2^\varepsilon) w\,' \de x \de y \de t\\
\\
+ \displaystyle \int_0^T\int_{\Omega\times Y^\ast}D_1(y)\mathcal{T}^{\ast}_{\varepsilon}(\nabla c_{1}^{\varepsilon})\cdot\mathcal{T}^{\ast}_{\varepsilon}(\nabla v_1^{\varepsilon})\, w \de x \de y  \de t + \displaystyle \int_0^T\int_{\Omega\times Y^\ast}D_2(y)\mathcal{T}^{\ast}_{\varepsilon}(\nabla c_{2}^{\varepsilon})\cdot\mathcal{T}^{\ast}_{\varepsilon}(\nabla v_2^{\varepsilon})\, w \de x \de y \de t\\
\\
+\displaystyle\int_0^T\int_{\Omega\times \Gamma}\mathcal{T}^{b}_{\varepsilon}\left(\dfrac{c_{1}
^{\varepsilon}-c_{2}^{\varepsilon}}{\varepsilon}H(c_3^{\varepsilon})\right)(\omega_1 \, \mathcal{T}^{\ast}_{\varepsilon}(\psi_1^\varepsilon)- \omega_2 \, \mathcal{T}^{\ast}_{\varepsilon}(\psi_2^\varepsilon)) \, w \de \sigma
_{x} \de y\de t\\
\\
=\displaystyle \int_0^T\int_{\Omega\times Y^\ast}\mathcal{T}^{\ast}_{\varepsilon}(F_{1}^{\varepsilon}(x, c_{1}^{\varepsilon}, c_{2}^{\varepsilon
}, c_3^{\varepsilon}))\mathcal{T}^{\ast}_{\varepsilon}(v_1^{\varepsilon}) \, w \de x \de y \de t \\
\\
+\displaystyle \int_0^T\int_{\Omega\times Y^\ast}\mathcal{T}^{\ast}_{\varepsilon}(F_{2}^{\varepsilon}(x, c_{1}^{\varepsilon}, c_{2}^{\varepsilon
}, c_3^{\varepsilon}))\mathcal{T}^{\ast}_{\varepsilon}(v_2^{\varepsilon}) \, w \de x \de y \de t.
\end{array}
\end{equation}
In order to obtain the macroscopic problem, we want to pass to the limit with $\varepsilon\rightarrow0$ in \eqref{unf3bis}, by using the convergence results for the microscopic solutions proved in the previous section. To this aim, let us analyze at first the nonlinear terms.\\
By assumption $(\mathbf{H}_2)$, due to the strong convergences \eqref{unfold-convconc}$_{iii}$ and  \eqref{unfold-convconc3}$_{iii}$ , for $i\in \{1,2,3\}$, one has
\begin{equation}\label{convFG}
\mathcal{T}^\ast_{\varepsilon}(F_i^{\varepsilon}(x,c_{1}^{\varepsilon},c_{2}%
^{\varepsilon}, c_3^{\varepsilon})) \rightarrow F_i(y,c_{1},c_{2}, c_{3})\text{  strongly in }%
L^{2}((0,T)\times\Omega\times Y^\ast).
\end{equation}
On the other hand, assumptions $(\mathbf{H}_3)_{2,4}$ imply that the function $H$ is globally Lipschitz-continuous and it holds $\vert H(s)\vert \leq L\vert s\vert$. Consequently $H(v)\in L^2(0,T;L^2(\Omega\times \Gamma))$ if $v\in L^2(0,T;L^2(\Omega\times\Gamma))$. Hence, the boundary term involving the function $H$ is well-defined, since the concentration fields are essentially bounded. Moreover, by \eqref{estimdifconc} and \eqref{normbound}, we get
\begin{equation}\label{boundest}
\begin{array}{l}
\left \|  \mathcal{T}^{b}_{\varepsilon }\left((c_{1}
^{\varepsilon}-c_{2}^{\varepsilon})H(c_3^{\varepsilon})\right)\right \|_ {L^2((0,T) \times \Omega \times \Gamma)} =
\\
\\
= \|  (\mathcal{T}^{b}_{\varepsilon } (c_1^\varepsilon) - \mathcal{T}^{b}_{\varepsilon } (c_2^\varepsilon)) \|_{L^2((0,T) \times \Omega \times \Gamma)}\|  H(\mathcal{T}^{b}_{\varepsilon } (c_3^\varepsilon) ) \|_{L^\infty((0,T) \times \Omega \times \Gamma)}  \leq C \, \varepsilon.
\end{array}
  \end{equation}
Therefore, there exists a function $U\in L^2((0,T) \times \Omega \times \Gamma)$ such that
\begin{equation}\label{convbound}
\mathcal{T}^{b}_{\varepsilon } \left( \frac{c_1^\varepsilon - c_2^\varepsilon}{\varepsilon}H(c_3^\varepsilon)\right) \rightharpoonup   U \quad \textrm{weakly  in } L^2((0,T)\times \Omega \times \Gamma).
\end{equation}
In order to identify the function $U$, let us define in $(0,T)\times\Omega^\ast_\varepsilon$ the function
\[
U^\varepsilon:=\dfrac{c_1^\varepsilon-c_2^\varepsilon}{\varepsilon}H(c_3^\varepsilon).
\]
Clearly, $U^\varepsilon \in L^2(0,T;H^1(\Omega^\ast_\varepsilon))$. Moreover, the {\it a priori} estimate  \eqref{estimdifconc} yields

\begin{equation}\label{estU}
\sqrt{\varepsilon}||U^\varepsilon||_{L^2((0,T)\times\Gamma_\varepsilon)}\leq C
\end{equation}
and since the chain rule applies according to \cite{Stampacchia}, Lemma 1.1 and $H'$ is bounded due to assumption $\mathcal{H}_3$, by  \eqref{estimconcLinf} we get
\begin{equation}\label{estgradU}
\begin{array}{l}
 \varepsilon ||\nabla U^\varepsilon ||_{L^2((0,T)\times\Omega^\ast_\varepsilon)}\leq ||\nabla c_1^\varepsilon-\nabla c_2^\varepsilon||_{L^2((0,T)\times\Omega^\ast_\varepsilon)}||H(c_3^\varepsilon)||_{L^\infty((0,T)\times\Omega^\ast_\varepsilon)}\\
\\
+||c_1^\varepsilon-c_2^\varepsilon||_{L^\infty((0,T)\times\Omega^\ast_\varepsilon)}||H'(c_3^\varepsilon)||_{L^\infty((0,T)\times\Omega^\ast_\varepsilon)}||\nabla c_3^\varepsilon||_{L^2((0,T)\times\Omega^\ast_\varepsilon)}\leq C.
\end{array}
\end{equation}
Estimates \eqref{estU} and \eqref{estgradU}, by using Lemma 6.1 in \cite{Conca}, imply
\[
||U^\varepsilon||_{L^2((0,T)\times\Omega^\ast_\varepsilon)}\leq C\left( \varepsilon ||\nabla U^\varepsilon ||_{L^2((0,T)\times\Omega^\ast_\varepsilon)}+ \sqrt{\varepsilon}||U^\varepsilon||_{L^2((0,T)\times\Gamma_\varepsilon)}\right)\leq C.
\]
Now, we can apply Proposition 2.8 in \cite{Donato-Yang*} and get the existence of $\widehat{U}\in L^2(0,T;L^2(\Omega;H^1(Y^\ast)))$ such that
\begin{equation}\label{convunfU}
\mathcal{T}^\ast_\varepsilon(U^\varepsilon)\rightharpoonup   \widehat{U} \quad \textrm{weakly  in } L^2((0,T)\times \Omega; H^1(Y^\ast))
\end{equation}
and
\begin{equation}\label{convunfgradU}
\mathcal{T}^\ast_\varepsilon(\varepsilon\nabla U^\varepsilon)\rightharpoonup   \nabla_y\widehat{U} \quad \textrm{weakly  in } L^2((0,T)\times \Omega\times Y^\ast).
\end{equation}
By the definition of the boundary unfolding operator $T^b_\varepsilon$ as the trace on $\Gamma$ of $T^\ast_\varepsilon$, due to the continuity of the trace operator, by \eqref{convbound} and \eqref{convunfU}, we deduce that $U$ is indeed the trace on $\Gamma$ of $\widehat{U}$.
The equality $U=\widehat{U}$ in $(0,T)\times\Omega\times\Gamma$ shows that it is enough to identify the function $\widehat{U}$. We are able to calculate the limit \eqref{convunfU}.
Indeed, for $\Psi\in \mathcal{D}((0,T)\times\Omega\times Y^\ast)^n$, by using the properties of the unfolding operator for perforated domains, we get
\begin{equation}\label{U*}
\begin{array}{c}
\displaystyle\int_0^T\int_{\Omega\times Y^\ast} \mathcal{T}^\ast_\varepsilon(\varepsilon \nabla U^\varepsilon)\Psi(x,y,t)\,\de x \de y \de t=\int_0^T\int_{\Omega\times Y^\ast} \mathcal{T}^\ast_\varepsilon \left((\nabla c_1^\varepsilon-\nabla c_2^{\varepsilon})H(c_3^\varepsilon)\right)\Psi(x,y,t)\,\de x \de y \de t\\
\\
\displaystyle+ \int_0^T\int_{\Omega\times Y^\ast} \mathcal{T}^\ast_\varepsilon \left((c_1^\varepsilon-c_2^{\varepsilon})H^{'}(c_3^\varepsilon)\nabla c_3^\varepsilon \right)\Psi(x,y,t)\,\de x \de y \de t\\
\\
=\displaystyle \int_0^T\int_{\Omega\times Y^\ast}  \left((\mathcal{T}^\ast_\varepsilon (\nabla c_1^\varepsilon)-T^\ast_\varepsilon (\nabla c_2^{\varepsilon}))H(\mathcal{T}^\ast_\varepsilon(c_3^\varepsilon)\right)\Psi(x,y,t)\,\de x \de y \de t
\\
\\
\displaystyle+\int_0^T\int_{\Omega\times Y^\ast} \left((\mathcal{T}^\ast_\varepsilon (c_1^\varepsilon)-T^\ast_\varepsilon (c_2^{\varepsilon}))T^\ast_\varepsilon (H^{'}(c_3^\varepsilon))T^\ast_\varepsilon (\nabla c_3^\varepsilon) \right)\Psi(x,y,t)\,\de x \de y \de t.
\end{array}
\end{equation}
We remark now that, according to \eqref{pos1}, \eqref{b1} and assumption $(\mathbf{H}_3)_3$, we have that
\begin{equation*}
\vert \mathcal{T}^{\ast}_\varepsilon(H^{'}(c_3^\varepsilon))\vert=\vert H^{'}(\mathcal{T}^{\ast}_\varepsilon(c_3^\varepsilon))\vert\leq L
\quad \textrm{in } L^\infty((0, T )\times \Omega\times Y^\ast),
\end{equation*}
hence by \eqref{unfold-convconc}iii) and \eqref{unfold-convconc3}ii), we get
\begin{equation}\label{boundLinf}
\int_0^T\int_{\Omega\times Y^\ast} \left((\mathcal{T}^\ast_\varepsilon (c_1^\varepsilon)-T^\ast_\varepsilon (c_2^{\varepsilon}))T^\ast_\varepsilon (H^{'}(c_3^\varepsilon))T^\ast_\varepsilon (\nabla c_3^\varepsilon) \right)\Psi(x,y,t)\,\de x \de y \de t\rightarrow 0.
\end{equation}
Using now convergence \eqref{unfold-convconc3}iii) and assumption $(\mathbf{H}_3)$ we obtain
\begin{equation}\label{convhh'}
\mathcal{T}^{\ast}_\varepsilon (H(c_3^\varepsilon))=H(\mathcal{T}^{\ast}_\varepsilon(c_3^\varepsilon)) \rightarrow   H(c_3)  \quad \textrm{strongly  in } L^2((0,T)\times \Omega\times Y^\ast).
\end{equation}
Hence, by \eqref{unfold-convconc}ii), \eqref{boundLinf} and \eqref{convhh'}, we obtain
\begin{equation*}
\begin{array}{c}
\displaystyle\int_0^T\int_{\Omega\times Y^{\ast}} \mathcal{T}^\ast_\varepsilon(\varepsilon \nabla U^\varepsilon)\Psi(x,y,t)\,\de x \de y \de t\rightarrow
\displaystyle\int_0^T\int_{\Omega\times Y} \left(\nabla_{y}\widehat{c}_{1}-\nabla_{y}\widehat{c}_{2})H(c_3\right)\Psi(x,y,t)\,\de x \de y \de t.
\end{array}
\end{equation*}
for any $\Psi\in \mathcal{D}((0,T)\times\Omega\times Y^\ast)^n$.\\
Comparing with \eqref{convunfgradU} and taking into account the fact that $H$ does not depend on the variable y and the functions $\widehat{c}_{1}$ and $\widehat{c}_{2}$ are defined up to an additive function depending on $t$ and $x$ only, we obtain $\widehat{U}=(\widehat{c}_{1}-\widehat{c}_{2})H(c_3)$.
So, finally
\begin{equation}\label{Ubound}
U=\widehat{U}=(\widehat{c}_{1}-\widehat{c}_{2})H(c_3).
\end{equation}

\bigskip
Thus, by \eqref{convtest}, \eqref{convFG}, \eqref{convbound} and \eqref{Ubound} we can pass to the limit as $\varepsilon\rightarrow0$ in \eqref{unf3bis} and we obtain:
\begin{equation*}
\begin{array}{l}
\displaystyle -2 \vert Y^\ast \vert \int_{0}^{T}\int_{\Omega}c \varphi w'\,\mathrm{d}x \mathrm{d}t+\int_{0}^{T}\int_{\Omega\times
Y^{\ast}}D_1(y)(\nabla c+\nabla_{y}\widehat{c_{1}})(\nabla\varphi+\omega_1\nabla_{y}%
\psi_1)\, w \de x \de y \de t\\
\\
\displaystyle +\int_{0}^{T}\int_{\Omega\times
Y^{\ast}}D_2(y)(\nabla c+\nabla_{y}\widehat{c_{2}})(\nabla\varphi+\omega_2\nabla_{y}%
\psi_2)\, w \de x \de y \de t\\
\\
+\displaystyle \int_{0}^{T}\int_{\Omega\times\Gamma} (\widehat{c_{1}}-\widehat{c_{2}}) H(c_3) (\omega_1\psi_1-\omega_2\psi_2) w \de x \de \sigma_{y} \de t\\
\\
=\displaystyle \int_{0}^{T} \int_{\Omega\times Y^{\ast}}F_{1}(y,c,c, c_{3})\varphi w\de x \de y \de t+\int_{0}^{T} \int_{\Omega\times Y^{\ast}}F_{2}(y,c,c, c_{3})\varphi w\de x \de y \de t,
\end{array}
\end{equation*}
for every $\varphi, \omega_i \in \mathcal{D}(\Omega)$, $\psi_i\in H_{\textrm{per}}^{1}(Y^\ast)$ and $w\in\mathcal{D}(0,T)$ which, by density, leads to the first equation in \eqref{probunf}. \\
In a similar way, by \eqref{convFG} and since due to assumption $(\mathbf{H}_4)$ and the strong convergences \eqref{unfold-convconc3}$_{iii}$ we get
\begin{equation}
\mathcal{T}_{\varepsilon}(G_3^{\varepsilon}(x, c_{1}^{\varepsilon},c_{2}%
^{\varepsilon}, c_3^{\varepsilon}))
\rightarrow G_3(y, c,c, c_3)\text{  strongly in }L^2((0,T)\times\Omega \times\Gamma),
\end{equation}
we obtain the equation governing the evolution of the concentration $c_3^\varepsilon$, as $\varepsilon \rightarrow 0$, which reads:
\begin{equation*}
\begin{array}{l}
\displaystyle -\vert Y^\ast \vert \int_{0}^{T}\int_{\Omega}c_{3}\varphi\partial_{t}%
w\,\mathrm{d}x \mathrm{d}t+\displaystyle \int_{0}^{T}\int_{\Omega\times
Y^{\ast}}D_3(y)(\nabla c_{3}+\nabla_{y}\widehat{c_{3}})(\nabla\varphi+\omega_3\nabla_{y}%
\psi_3)\, w \de x \de y \de t\\
\\
=\displaystyle \int_{0}^{T}\int_{\Omega\times\Gamma}G_{3}(y,c,c, c_3)\varphi w\mathrm{d}x\mathrm{d}\sigma_{y}\mathrm{d}t+\int_{0}^{T}%
\int_{\Omega\times Y^{\ast}}F_{3}(y,c,c,c_{3})\varphi w \de x \de y \de t
\end{array}
\end{equation*}
for every $\varphi, \omega_i \in \mathcal{D}(\Omega)$, $\psi_i\in H_{\textrm{per}}^{1}(Y^\ast)$ and $w\in\mathcal{D}(0,T)$ which, by density, gives the second equation in \eqref{probunf}. \\

\noindent The initial condition is obtained in a standard way.
\end{proof}\\

\begin{corollary}\label{teohom} Let $(c_{1}^{\varepsilon}, c_{2}^{\varepsilon}, c_{3}^{\varepsilon})$ be the solution of system  \eqref{weakconcentration} - \eqref{weakinitial}. Then, if the assumptions $(\mathbf{H}_1) \div (\mathbf{H}_5)$ hold, there exist $c$ and $c_{3}$ in
$L^{2}(0,T;H_0^{1}(\Omega))$ such that convergences \eqref{unfold-convconc} and \eqref{unfold-convconc3} hold and the pair $(c, c_3)$ is the unique solution of the coupled system
\begin{equation}\label{homsystem}
\left\{
\begin{array}{ll}
2  \dfrac{\partial c }{\partial t}-\operatorname{div}
(B(c_3)\nabla c)=  \mathcal{M}_{Y^{\ast}}
(F_{1}(\cdot,c,c,c_{3}))+\mathcal{M}_{Y^{\ast}}
(F_{2}(\cdot,c,c,c_{3}))& \text{\rm in }(0,T)\times\Omega\\
\\
\dfrac{\partial c_3 }{\partial t}-\operatorname{div}
(D^{0}\nabla c_3)= \mathcal{M}_{Y^{\ast}}
(F_{3}(\cdot,c,c,c_{3}))+\dfrac{\vert \Gamma \vert }{\vert Y^{\ast}\vert}\mathcal{M}_{\Gamma}
(G_{3}(\cdot,c,c,c_{3}))& \text{\rm in }(0,T)\times\Omega\\
\\
c(x,0)=(c_1^0+c_2^0)/2 & \text{\rm in } \Omega\\
\\
c_3(x,0)=c_3^0  & \text{\rm in } \Omega.
\end{array}
\right.
\end{equation}
The positive definite constant homogenized diffusion matrix $D^0$ is given, for $i,j\in \{1,\dots,n\}$, by its entries
\begin{equation}\label{matrixA}
D^0_{ij}=\mathcal{M}_{Y^\ast}\left((D_3(y))_{ij}-\sum_{k=1}^{n}(D_3(y))_{ik}\dfrac{\partial\chi^j}{\partial y_k}(y)\right),
\end{equation}
where, for $j\in \{1,\dots,n\}$, $\chi^j \in H^1_{\textrm{per}}(Y^{\ast})/{\Bbb R}$ verifies the local problem
\begin{equation} \label{cellprobl1}
\left\{
\begin{array}{ll}
-\displaystyle {\rm div}_y \left(D_3(y)(\nabla_y \chi^{j}-\mathbf{e}_{j})\right)=0 & \text{\rm in }Y^{\ast},\\
\\
D_3(y)(\nabla_y\chi^{j}-\mathbf{e}_{j})\cdot\nu=0 & \text{\rm on }\Gamma,\\
\\
\chi^{j} \text{ \rm Y-periodic }\,\text{\rm such that }\mathcal{M}_{Y^\ast}(\chi^j)=0,
\end{array}
\right.
\end{equation}
where $\nu$ is the outward unit normal to the boundary $\Gamma$.\\
For every $s\in\mathbb{R}$, the non-constant matrix $B(s)$ is defined, for $i,j\in \{1,\dots,n\}$, by its entries
\begin{equation}\label{matrixA0}
(B(s))_{ij}=\mathcal{M}_{Y^\ast}\left((D_1(y))_{ij}-\sum_{k=1}^{n}(D_1(y))_{ik}\dfrac{\partial\chi_1^j}{\partial y_k}(y,s)\right)+\mathcal{M}_{Y^\ast}\left((D_2(y))_{ij}-\sum_{k=1}^{n}(D_2(y))_{ik}\dfrac{\partial\chi_2^j}{\partial y_k}(y,s)\right),
\end{equation}
where the pair $(\chi_1^j(\cdot, s), \, \chi_2^j(\cdot, s))\in H^1_{\textrm{per}}(Y^{\ast})/{\Bbb R}\times H^1_{\textrm{per}}(Y^{\ast})/{\Bbb R}$ is, up to the addition of the same constant to $\chi_1^j(\cdot, s)$ and $\chi_2^j(\cdot, s)$, the unique solution of the local problem
\begin{equation} \label{cellprobl2}
\left\{
\begin{array}{ll}
- \displaystyle{\rm div}_{y} \left(D_1(y)(\nabla_{y} \chi_1^j-\mathbf{e}_{j})\right)=0  & \text{\rm in   } Y^\ast, \\
\\
- \displaystyle{\rm div}_{y} \left(D_2(y)(\nabla_{y} \chi_2^j-\mathbf{e}_{j})\right)=0 & \text{\rm in   } Y^\ast, \\
\\
D_1(y)(\nabla_{y} \chi_1^j-\mathbf{e}_{j})\cdot \nu=-H(s)(\chi_1^j -\chi_2^j) & \text{\rm on   } \Gamma, \\
\\
D_2(y)(\nabla_{y} \chi_2^j-\mathbf{e}_{j}) \cdot \nu= H(s)(\chi_1^j -\chi_2^j) &   \text{\rm on   } \Gamma,\\
\\
\chi_i^{j} \text{ \rm Y-periodic },\, i\in\{1,2 \}\,\text{\rm  and such that }\mathcal{M}_{Y^\ast}(\chi_1^j)=0,
\end{array}
\right.
\end{equation}
where $\nu$ is the outward unit normal to the boundary $\Gamma$.

\end{corollary}
\begin{proof}
Taking $\varphi=0$ in the second equation in \eqref{probunf}, we get
\begin{equation*}
\displaystyle\int_{\Omega\times Y^{\ast}}D_3(y)(\nabla c_3+\nabla_{y}\widehat{c_{3}})\nabla_{y}\Psi_3
\de x \de y =0\,\text{ in }\mathcal{D}'(0,T).
\end{equation*}
 for all $\Psi_3\in L^2(\Omega;H_{per}^{1}(Y^\ast))$. Hence, for a.e. $t\in (0,T)$, we get
\begin{equation}\label{probconc}
\left\{
\begin{array}{ll}
-\hbox{div}_{y}\left(D_3(y)\nabla_y\widehat{c_3}\right)=\hbox{div}_{y}(D_3(y)\nabla c_3) &\text{ in }\Omega\times Y^{\ast},\\[2mm]
D_3(y)\nabla_{y}\widehat{c_3}\cdot\nu=-D_3(y)\nabla c_3\cdot\nu&\text{ on
}\Omega\times\Gamma,\\[2mm]
\widehat{c_3}\quad\text{periodic in }y.
\end{array}
\right.
\end{equation}
By linearity, we get, for a.e. $t\in (0,T)$ and $(x,y)\in \Omega\times Y^\ast$,
\begin{equation}\label{cihat}
\widehat{c_{3}} (t,x,y)=-\sum\limits_{j=1}^{n}\chi^{j}(y)\frac{\partial
c_{3}}{\partial x_{j}}(t,x),
\end{equation}
where $\chi^{j},\,j\in \{1,\dots,n\}$, are the solutions of the local
problems \eqref{cellprobl1}.

Taking $\Psi_3=0$ in the second equation of \eqref{probunf}, we get
\begin{equation*}
\begin{array}{l}
\vert Y^\ast \vert \displaystyle \langle\partial_t c_3, \varphi\rangle_{\Omega}+ \int_{\Omega\times
Y^{\ast}}D_3(y)(\nabla c_{3}+\nabla_{y}\widehat{c_{3}})\nabla\varphi \de x \de y\\
\\
=\displaystyle \int_{\Omega\times\Gamma}G_{3}(y,c,c, c_{3})\varphi\mathrm{d}x\mathrm{d}\sigma_{y}+
\int_{\Omega\times Y^{\ast}}F_{3}(y,c,c, c_{3})\varphi \de x \de y ,\text{ in }\mathcal{D}'(0,T).
\end{array}
\end{equation*}
Replacing $\widehat{c}_3$ given by \eqref{cihat} in the previous equality, we obtain
\begin{equation*}
\begin{array}{l}
\vert Y^\ast \vert \displaystyle \langle\partial_t c_3, \varphi\rangle_{\Omega}+\int_{\Omega}\sum\limits_{i=1}^{n}\sum\limits_{j=1}^{n} \left( \int_{Y^{\ast}}\left((D_3(y))_{ij}-\sum_{k=1}^{n}(D_3(y))_{ik}\dfrac{\partial \chi^{j}}{
\partial y_{k}}\left( y\right) \right) \de y\right)\dfrac{\partial c_3}{\partial x_{j}}\dfrac{\partial
\varphi}{\partial x_{i}}	\,\de x
\\
 \\
 =\displaystyle\int_{\Omega\times\Gamma}G_{3}(y,c,c, c_{3})\varphi \de x \de \sigma_{y}+
\int_{\Omega\times Y^{\ast}}F_{3}(y,c,c, c_{3})\varphi \de x \de y\text{ in }\mathcal{D}'(0,T)
 \end{array}
\end{equation*}
for all $\varphi \in H_0^{1}\left( \Omega \right)$, which means that $(c, c_3)$ satisfies the following problem
\begin{equation*}
\left\{
\begin{array}{ll}
\displaystyle \dfrac{\partial c_{3}}{\partial t}-\sum\limits_{i=1}^{n}\dfrac{\partial
}{\partial x_{i}}\sum\limits_{j=1}^{n} \left(\dfrac{1}{|Y^\ast|} \int_{Y^{\ast}}\left((D_3(y))_{ij}-\sum_{k=1}^{n}(D_3(y))_{ik}\dfrac{\partial \chi^{j}}{
\partial y_{i}}\left( y\right) \right) \de y\right)\dfrac{\partial c_3}{\partial x_{j}}\\
\\
\\= \dfrac{|\Gamma|}{|Y^\ast|}\mathcal{M}_{\Gamma}
(G_{3}(\cdot,c,c, c_{3}))+\mathcal{M}_{Y^{\ast}}
(F_{3}(\cdot,c,c, c_{3}))& \text{\rm in }(0,T)\times\Omega,\\
\\
c_3=0& \text{\rm on }(0,T)\times\partial\Omega.
\end{array}%
\right.
\end{equation*}
Thus, we are led to the homogenized equation for $c_3$ in \eqref{homsystem}, where the constant matrix $D^0$ is defined through \eqref{matrixA}. \\

Let us take now $\varphi=0$ in the first equation in \eqref{probunf}. We have:
\begin{equation*}
\begin{array}{ll}
\displaystyle \int_{\Omega\times Y^{\ast}}D_1(y)(\nabla c+\nabla_{y}\widehat{c_{1}})\nabla_{y}\Psi_1\,\de x \de y+\displaystyle \int_{\Omega\times Y^{\ast}}D_2(y)(\nabla c+\nabla_{y}\widehat{c_{2}})\nabla_{y}\Psi_2\,\de x \de y \\
\\
+\displaystyle \int_{\Omega\times\Gamma} (\widehat{c_{1}}-\widehat{c_{2}}) H(c_3) (\Psi_1-\Psi_2) \de x \de \sigma_{y}=0\text{ in }\mathcal{D}'(0,T),
\end{array}%
\end{equation*}
for all $\Psi_i\in L^2(\Omega;H_{per}^{1}(Y^\ast))$, $i\in\{1,2\}$. Hence, for a.e. $t\in (0,T)$, we have
\begin{equation}\label{probconc}
\left\{
\begin{array}{ll}
-\hbox{div}_{y}\left(D_1(y)\nabla_y\widehat{c_1}\right)=\hbox{div}_{y}(D_1(y)\nabla c) &\text{ in }\Omega\times Y^{\ast},\\
\\
-\hbox{div}_{y}\left(D_2(y) \nabla_y \widehat{c_2}\right)=\hbox{div}_{y}(D_2(y)\nabla c)&\text{ in }\Omega\times Y^{\ast},\\
\\
D_1(y)\nabla_{y}\widehat{c_1}\cdot\nu=-D_1(y)\nabla c\cdot \nu -H(c_3)(\widehat{c_1}-\widehat{c_2})&\text{ on
}\Omega\times\Gamma,\\
\\
D_2(y)\nabla_{y}\widehat{c_2}\cdot\nu=-D_2(y)\nabla c\cdot \nu + H(c_3)(\widehat{c_1}-\widehat{c_2})&\text{ on
}\Omega\times\Gamma,\\
\\
\widehat{c_i}\ \text{periodic in }y,\,i\in\{1,2\}.
\end{array}
\right.
\end{equation}
Then, acting as previously, for a.e. $t\in (0,T)$ and $(x,y)\in \Omega\times Y^\ast$, we get
\begin{equation}\label{uhat1}
 \widehat{c}_1(t,x,y)=-\displaystyle \sum\limits_{j=1}^n \chi_1^j (y, c_3(t,x)) \displaystyle \frac{\partial c}{\partial x_j}(t,x) \quad \textrm{in  }
 (0,T) \times \Omega \times Y^{\ast},
\end{equation}
\begin{equation}\label{uhat2}
 \widehat{c}_2(t,x,y)=-\displaystyle \sum\limits_{j=1}^n \chi_2^j (y, c_3(t,x)) \displaystyle \frac{\partial c}{\partial x_j}(t,x)  \quad \textrm{in  }
 (0,T) \times \Omega \times Y^{\ast},
\end{equation}
where $\chi_k^j(c_3)$, $k\in \{1,2\}$ and $j\in \{1,\dots,n\}$, up to the addition of the same constant,
are the unique solutions of the local problems \eqref{cellprobl2}. 

Taking $\Psi_1=\Psi_2=0$ in the first equation of \eqref{probunf}, we get
\begin{equation}
\begin{array}{l}
2\displaystyle \vert Y^\ast \vert \, \langle\partial_t c, \varphi\rangle_{\Omega}+
\displaystyle \int_{\Omega\times Y^{\ast}}D_1(y)(\nabla c+\nabla_{y}\widehat{c_{1}})\nabla\varphi \de x \de y+\displaystyle \int_{\Omega\times Y^{\ast}}D_2(y)(\nabla c+\nabla_{y}\widehat{c_{2}})\nabla\varphi \de x \de y \\
\\
\displaystyle = \int_{\Omega\times Y^{\ast}}F_{1}(y,c,c, c_{3})\varphi \de x \de y+\int_{\Omega\times Y^{\ast}}F_{2}(y,c,c, c_{3})\varphi \de x \de y,\text{ in }\mathcal{D}'(0,T).
\end{array}
\end{equation}
Replacing $\widehat{c}_i$, $i\in\{1,2\}$, given by \eqref{uhat1} and \eqref{uhat2} in the previous equality, we obtain
\begin{equation*}
\begin{array}{l}
2\vert Y^\ast \vert \displaystyle \langle\partial_t c, \varphi\rangle_{\Omega}+\int_{\Omega}\sum\limits_{i=1}^{n}\sum\limits_{j=1}^{n} \left( \int_{Y^{\ast}}\left((D_1(y))_{ij}-\sum_{k=1}^{n}(D_1(y))_{ik}\dfrac{\partial \chi_1^{j}}{
\partial y_{k}}\left( y\right) \right) \de y\right)\dfrac{\partial c}{\partial x_{j}}\dfrac{\partial
\varphi}{\partial x_{i}}	\,\de x
\\
 \\
 + \displaystyle\int_{\Omega}\sum\limits_{i=1}^{n}\sum\limits_{j=1}^{n} \left( \int_{Y^{\ast}}\left((D_2(y))_{ij}-\sum_{k=1}^{n}(D_2(y))_{ik}\dfrac{\partial \chi_2^{j}}{
\partial y_{k}}\left( y\right) \right) \de y\right)\dfrac{\partial c}{\partial x_{j}}\dfrac{\partial
\varphi}{\partial x_{i}}	\,\de x\\
\\
 =\displaystyle\int_{\Omega\times Y^{\ast}}F_{1}(y,c,c, c_{3})\varphi \de x \de y+
\int_{\Omega\times Y^{\ast}}F_{2}(y,c,c, c_{3})\varphi \de x \de y\text{ in }\mathcal{D}'(0,T)
 \end{array}
\end{equation*}
for all $\varphi \in H_0^{1}\left( \Omega \right)$, which means that $(c, c_3)$ satisfies the following problem
\begin{equation*}
\left\{
\begin{array}{ll}
2\displaystyle \dfrac{\partial c}{\partial t}-\sum\limits_{i=1}^{n}\dfrac{\partial
}{\partial x_{i}}\sum\limits_{j=1}^{n} \sum\limits_{l=1}^{2}\left(\dfrac{1}{|Y^\ast|} \int_{Y^{\ast}}\left((D_l(y))_{ij}-\sum_{k=1}^{n}(D_l(y))_{ik}\dfrac{\partial \chi_l^{j}}{
\partial y_{i}}\left( y\right) \right) \de y\right)\dfrac{\partial c}{\partial x_{j}}\\
\\
\\= \mathcal{M}_{Y^{\ast}}
(F_{1}(\cdot,c,c, c_{3}))+\mathcal{M}_{Y^{\ast}}
(F_{2}(\cdot,c,c, c_{3}))& \text{\rm in }(0,T)\times\Omega,\\
\\
c=0& \text{\rm on }(0,T)\times\partial\Omega.
\end{array}%
\right.
\end{equation*}
Thus, we are led to the homogenized equation for $c$ in \eqref{homsystem}, where the non-constant matrix $B(c_3)$ is defined through \eqref{matrixA0}. \\

We remark that, for any given $s \in \mathbb{R}$,
under the assumptions $(\mathbf{H}_1)$ and $(\mathbf{H}_3)$, one can prove, by using Lax-Milgram theorem, the existence and the
uniqueness, up to a constant, of a solution $(\chi_1^j(\cdot,s), \chi_2^j(\cdot, s)) \in H^1_{\textrm{per}}(Y^{\ast})/{\mathbb R} \times
H^1_{\textrm{per}}(Y^{\ast})/{\mathbb R} $ for problem
\eqref{cellprobl2}. The value of the constant needs to be the same.

The constant homogenized matrix $D^0$ is positive definite. By assumption $(\mathbf{H}_3)$, the non-constant homogenized matrix $B$ is uniformly coercive and uniformly bounded from above. Therefore, as in \cite{Allaire}, this implies the uniqueness of the solution of the limit problem and, thus, all the above convergence results hold for the whole sequences. So, the couple $(c, c_3)$ is the unique solution of problem \eqref{homsystem}, where the matrices $D^0$ and $B$ are defined by \eqref{matrixA} and \eqref{matrixA0}, respectively. \end{proof}\\

\begin{remark}\label{rem5}
We point out that in system \eqref{homsystem} we have two homogenized matrices, a standard one $D^0$ and a non-constant one $B$, generated by the special coupling and scalings of the boundary terms in the microscopic system \eqref{eqmicro}. Moreover, it is easy to prove that the positive definite constant homogenized diffusion matrix $D^0$ can be written also as
\begin{equation}\label{matrixAbis}
D^0_{ij}=\dfrac{1}{\vert Y^\ast \vert}\displaystyle \int_{Y^{\ast}} D_3(y)\left(e_i-\nabla_y \chi^i \right) \left(e_j-\nabla_y \chi^j\right)\de y\,\,\,\text{\rm for }i,j\in \{1,\dots,n\}
\end{equation}
and, for every $s\in \mathbb{R}$, the non-constant dispersion matrix $B(s)$ can also be written as
\begin{equation}\label{matrixA0bis}
\begin{array}{c}
(B(s))_{ij}=\dfrac{1}{\vert Y^\ast \vert}\displaystyle \int_{Y^{\ast}} D_1(y)\left(e_i-\nabla_y \chi_1^i\right)
\left(e_j-\nabla_y \chi_1^j\right)\de y+\dfrac{1}{\vert Y^\ast \vert}\displaystyle \int_{Y^{\ast}} D_2(y)\left(e_i-\nabla_y \chi_2^i\right)
\left(e_j-\nabla_y \chi_2^j\right)\de y\\
\\
+\dfrac{1}{\vert Y^\ast \vert}\displaystyle \int_{\Gamma} H(s) \left(\chi_1^i -
\chi_2^i\right)\left(\chi_1^j -\chi_2^j\right) \de \sigma_y,\,\,\,\text{\rm for }i,j\in \{1,\dots,n\}.
\end{array}
\end{equation}
Here, for $i\in \{1,\dots,n\}$, $\chi^i \in H^1_{\textrm{per}}(Y^{\ast})/{\Bbb R}$ verifies the local problem \eqref{cellprobl1}, while for every $s\in\mathbb{R}$, $\chi_1^i(\cdot,s), \, \chi_2^i (\cdot, s)\in H^1_{\textrm{per}}(Y^{\ast})/{\Bbb R}$ are, up to the addition of the same constant, the unique solutions of the local problem \eqref{cellprobl2}.
\end{remark}
\begin{remark}\label{remA0}
The fact that the non-linearity with respect to $c_3^\varepsilon$ passes from the reaction term, at the microscopic level, to the diffusion term, at the macroscopic one (see \eqref{homsystem}), is a manifestation of the different scaling of reaction terms on the boundary (see \eqref{eqmicro}). In other words, the diffusion of concentration $c_3^\varepsilon$ on the boundary is slower than the ones of $c_1^\varepsilon$ and $c_2^\varepsilon$. In particular, we observe that for small values of the concentration $c_3$, due to the Lipschitz continuity property of the function $H$ and the fact that $H(0)=0$, the cell problems \eqref{cellprobl2} become decoupled and the dispersion tensor $B$ is the sum of two matrices of the same type of $D^0$ carrying at the macroscopic level the contributions of each concentration $c^\varepsilon_i$, $i\in \{1,2\}$. On the other hand, for large values of $c_3$, if in particular $H$ is of Langmuir type (see \eqref{exampleH}), as in \cite{Allaire}, we have the saturation effect of the Langmuir isotherm. In this case, since by \eqref{exampleH} the limit of $H(s)$ for $s \to \infty$ is the constant $\dfrac{a}{b}$, if we assume that $c_3$ goes to infinity, then the cell problem corresponding to such an infinite reaction limit is no longer dependent on the homogenized solution $c_3$, but it remains coupled.
Of course, other mathematical settings could be considered in place of \eqref{formG} for getting such an effect and they will be dealt with in a forecoming paper.
\end{remark}
\begin{remark}
If in the microscopic system \eqref{eqmicro} all the reaction terms on $\Gamma^{\varepsilon}$ are scaled with $\varepsilon$, as, for instance, in \cite{Gahn}, i.e. if we consider the system
\begin{equation}
\left\{
\begin{array} {ll}
\partial_{t}c_{i}^{\varepsilon}-\operatorname{div}(D_i^\varepsilon \nabla c_{i}^{\varepsilon
})=F_{i}^{\varepsilon}%
(x,c_{1}^{\varepsilon},c_{2}^{\varepsilon}, c_{3}^{\varepsilon})&\text{ in
}(0,T)\times\Omega^{\ast}_{\varepsilon},\,\, i\in \{1,2,3\},\\
\\
D_i^\varepsilon \nabla c_{i}^{\varepsilon}\cdot\nu^\varepsilon=\varepsilon\, G_{i}(c_{1}^{\varepsilon}%
,c_{2}^{\varepsilon}, c_{3}^{\varepsilon})&\text{ on }(0,T)\times
\Gamma^{\varepsilon},\, \, i\in \{1,2\},\\
\\
D_3^\varepsilon \nabla c_{3}^{\varepsilon}\cdot\nu^\varepsilon=\varepsilon\, G_{3}^{\varepsilon}(x,c_{1}^{\varepsilon}%
,c_{2}^{\varepsilon}, c_{3}^{\varepsilon})&\text{ on }(0,T)\times
\Gamma^{\varepsilon},\\
\\
c_{i}^{\varepsilon}=0 & \text{ on }(0,T)\times\partial\Omega,\, \,  i\in \{1,2,3\},\\
\\
c_{i}^{\varepsilon}(0)=c_{i}^0 & \text{ in }\Omega^{\ast}_{\varepsilon},\, \,  i\in \{1,2,3\},
\end{array}
\right.  \label{eqmicro-v}%
\end{equation}
then it is not difficult to see that there exist $c_{i}\in
L^{2}(0,T;H_0^{1}(\Omega))$ and $\widehat{c}_{i}\in L^{2}((0,T)\times\Omega;H_{per}^{1}(Y^\ast)/\Bbb R)$, $i\in \{1,2,3\}$, such that
\begin{equation} \label{conv-v}
\left\{
\begin{array}{ll}
\mathcal{T}^{\ast}_{\varepsilon}(c_{i}^{\varepsilon})\rightharpoonup c_{i}&\text{\rm weakly  in }L^{2}((0,T)\times\Omega;H^{1}(Y^\ast)),\\[2mm]
\mathcal{T}^{\ast}_{\varepsilon}(\nabla c_{i}^{\varepsilon})\rightharpoonup\nabla c_{i}+\nabla_{y}\widehat{c}_{i}&\text{\rm weakly in }L^{2}((0,T)\times
\Omega\times Y^\ast),\\ [2mm]
\mathcal{T}^\ast_{\varepsilon } (c_i^\varepsilon) \rightarrow c_i & \text{\rm strongly in } L^2((0,T)\times \Omega\times Y^\ast).
\end{array}
\right.
\end{equation}
In this case, the limit function $(c_1, c_2, c_3)$ in \eqref{conv-v} is the unique solution of the following system:
\[
\left\{
\begin{array}{ll}
\dfrac{\partial c_1 }{\partial t}-\operatorname{div}
(D_1^0\nabla c_1)=  \mathcal{M}_{Y^{\ast}}
(F_{1}(\cdot,c_{1},c_{2},c_{3}))-(c_1-c_2)\, \dfrac{\vert \Gamma \vert}{\vert Y^\ast \vert}\,  H(c_{3})& \text{\rm in }(0,T)\times\Omega,\\
\\
\dfrac{\partial c_2 }{\partial t}-\operatorname{div}
(D_2^0\nabla c_2)=  \, \mathcal{M}_{Y^{\ast}}
(F_{2}(\cdot,c_{1},c_{2},c_{3}))+(c_1-c_2) \, \dfrac{\vert \Gamma \vert}{\vert Y^\ast \vert} \, H(c_{3})& \text{\rm in }(0,T)\times\Omega,\\
\\
\dfrac{\partial c_3 }{\partial t}-\operatorname{div}
(D_3^0\nabla c_3)= \mathcal{M}_{Y^{\ast}}
(F_{3}(\cdot,c_{1},c_{2},c_{3}))+\dfrac{\vert \Gamma \vert}{\vert Y^\ast \vert} \mathcal{M}_{\Gamma}
(G_{3}(\cdot,c_{1},c_{2},c_{3}))& \text{\rm in }(0,T)\times\Omega, \\
\\
c_i=0  & \text{\rm on } (0,T)\times\partial\Omega,\\
\\
c_i(x,0)=c_i^0  & \text{\rm in } \Omega,
\end{array}
\right.
\]
where the entries of the positive definite constant homogenized diffusion matrices $D_i^0$, $i\in\{1, 2, 3\}$, are of the same type of  \eqref{matrixA}.
\end{remark}

\bigskip

{\bf Acknowledgments.} This paper was completed during the visit of the last author at the University of Sannio, Department of Engineering, whose warm hospitality and support are gratefully acknowledged. The work was supported by the grant FFABR of MIUR. G.C. and C.P. are members of GNAMPA of INDAM.

\end{document}